\documentclass[reqno,12pt,twosided,a4paper]{amsart}
\usepackage{amssymb,amsfonts,latexsym,mathrsfs,amsmath}
\usepackage[all]{xy}
\usepackage{array}

\usepackage{amstext}
\usepackage{amssymb}
\usepackage{amsfonts}
\usepackage{amsthm}
\usepackage{amsopn}
\usepackage{amsbsy}
\usepackage{layout}
\usepackage{color}

\newtheorem{theorem}{Theorem}[section]
\newtheorem{lemma}[theorem]{Lemma}
\newtheorem{proposition}[theorem]{Proposition}
\newtheorem{corollary}[theorem]{Corollary}

\theoremstyle{definition}
\newtheorem{definition}[theorem]{Definition}

\newtheorem{remark}[theorem]{Remark}

\newcommand{\Z}{\mathbb{Z}}
\newcommand{\Q}{\mathbb{Q}}

\newcommand{\C}{\mathcal{C}}
\newcommand{\D}{\mathcal{D}}

\newcommand{\T}{\mathcal{T}}
\newcommand{\ot}{\otimes}
\newcommand{\mult}{\times}
\newcommand{\til}{\tilde}
\newcommand{\con}{\backsim}
\newcommand{\one}{\boldsymbol{1}}
\newcommand{\ti}{\widetilde}
\newcommand{\Ga}{\Gamma}

\newcommand{\A}{{\mathcal{A}}}
\newcommand{\B}{{\mathcal{B}}}
\newcommand{\ra}{\rightarrow}
\newcommand{\wb}{\overline}
\renewcommand{\l}{\ell}

\newcommand{\X}{\mathcal{X}}

\begin{document}
\title[Descent via categories]{Descent, fields of invariants, and generic forms via symmetric monoidal categories}
\author{Ehud Meir}
\address{E. Meir: Department of Mathematical Sciences, University of Copenhagen, Universitetsparken 5, DK-2100, Denmark}
\email{meirehud@gmail.com}

\begin{abstract}
Let $W$ be a finite dimensional algebraic structure (e.g. an algebra) over a field $K$ of characteristic zero.
We study forms of $W$ by using Deligne's Theory of symmetric monoidal categories.
We construct a category $\C_W$, which gives rise to a subfield $K_0\subseteq K$, which we call the field of invariants of $W$.
This field will be contained in any subfield of $K$ over which $W$ has a form.
The category $\C_W$ is a $K_0$-form of $Rep_{\bar{K}}(Aut(W))$, and we use it to construct a generic form $\ti{W}$ over a commutative $K_0$-algebra $B_W$ 
(so that forms of $W$ are exactly the specializations of $\ti{W}$).
This generalizes some generic constructions for central simple algebras and for $H$-comodule algebras. 
We give some concrete examples arising from associative algebras and $H$-comodule algebras.
As an application, we also explain how one can use the construction to classify two-cocycles on some finite dimensional Hopf algebras.
\end{abstract}
\maketitle
\begin{section}{Introduction}\label{introduction}
Let $W$ be a finite dimensional algebraic structure (e.g. an algebra, a Hopf algebra, a comodule algebra, a module over a given algebra et cetera) 
defined over a field $K$ of characteristic zero (for the clarity of the exposition we will assume that $K$ is algebraically closed).
In this paper we will discuss the following questions:\\
1. Over what subfields $K_1$ of $K$ can $W$ be defined (i.e. over what subfields of $K$ does $W$ have a \textit{form})?\\
2. In case $W$ can be defined over a subfield $K_1$ of $K$, what are all the forms of $W$ over this subfield?\\
3. What scalar invariants does the isomorphism type of $W$ have?\\
These questions are very hard in general. 
For example, the forms of $M_n(K)$ are all the central simple algebras of dimension $n^2$ over some subfield of $K$.

In this paper we will address these questions by using tools from the theory of symmetric monoidal categories,
and by using a theorem of Deligne, which says that every symmetric monoidal category of a specific type is the representation category of an algebraic group.
We will construct an abelian rigid symmetric monoidal category $\C_W$ (which we shall call the fundamental category of $W$)
which contains a structure $\wb{W}$ of the same type as $W$, 
and an exact faithful symmetric monoidal functor $F:\C_W\ra Vec_K$ such that $F(\wb{W})=W$. 
This category will be defined over a subfield $K_0$ of $K$.
In a sense which will be explained later, $\C_W$ will be the symmetric rigid monoidal category ``generated by $W$'' and
the field $K_0$ will be the ``field of invariants'' of $W$.
We will prove in Section \ref{construction} that $(\C_W,\wb{W},F)$ satisfies the following universal property:
\begin{theorem}\label{universalproperty}
Let $K\subseteq T$ be an extension field, let $\A$ be an abelian symmetric monoidal rigid category, and let $G:\A\ra Vec_T$ 
be an exact additive faithful symmetric monoidal functor.
Assume that there is a structure $Z$ in $\A$ such that 
the structures $G(Z)$ and $W\ot_K T$ are isomorphic.
Then there exists a unique (up to isomorphism) exact faithful symmetric monoidal functor $\ti{F}:\C_W\ra \A$ 
such that $\ti{F}(\wb{W}) = Z$ (where equality here means equality of structures), 
and such that $G\ti{F}\cong i_{K,T}F$ where $i_{K,T}:Vec_K\ra Vec_T$ is the extension of scalars functor. 
\end{theorem}
We will then use Deligne's Theory on symmetric monoidal categories to study the category $\C_W$.
In Section \ref{categoryform} we will prove the following result:
\begin{theorem}\label{main1}
The category $\C_W$ is a $K_0$-form of the $K$-linear category $Rep_K-G$, where $G$ is the algebraic group of automorphisms of $W$.
For any subfield $K_0\subseteq K_1\subseteq K$,
we have a one-to-one correspondence between forms of $W$ over $K_1$ and isomorphism classes of fiber functors $F:\C_W\otimes_{K_0} K_1\rightarrow  Vec_{K_1}$.
\end{theorem}

Theorem \ref{main1} gives us a description of forms of $W$ in terms of fiber functors on the category $\C_W$ 
(by a fiber functor we mean here an exact additive symmetric monoidal functor, see Section \ref{prel}).
However, we would like to have a more concrete answer to Questions 1 2 and 3 above.
In Section \ref{generic} we will use Deligne's Theory in order to construct a ``generic form'' of $W$, which specializes
to all forms of $W$, and only to forms of $W$.
More precisely, we will prove the following result:
\begin{theorem}\label{main2}
There exists a commutative $K_0$-algebra $B_W$
and a $B_W$-structure $\ti{W}$ of the same type as $W$ such that:\\
1. As a $B_W$-module, $\ti{W}$ is free of rank $dim_K W$.\\
2. If $\phi:B_W\rightarrow K_1$ is a homomorphism of $K_0$-algebras from $B_W$ to an extension field $K_1$ of $K_0$, 
then $W_{\phi}:=\ti{W}\ot_{B_W} K_1$ is a form of $W$ over $K_1$.\\
3. Every form of $W$ is of the form $W_{\phi}$ for some $\phi$.\\
4. The algebra $B_W$ has no zero divisors. If the group $G$ is reductive, then the algebra $B_W$ can be chosen to be finitely generated.
In this case, the structure $W$ will have a form over a finite extension of $K_0$. 
\end{theorem}
Thus, $W$ will have a form over $K_1$ if and only if there exists a homomorphism $B_W\rightarrow K_1$,
and any form of $W$ over $K_1$ will be of the form $W_{\phi}$ for some homomorphism $\phi:B_W\rightarrow K_1$.

The field $K_0$ contains elements which must be contained in any field over which $W$ has a form.
We will prove in Section \ref{sec:galois} the following characterization of $K_0$ in case the extension $K/K_0$ is Galois. 
\begin{theorem}\label{main4}
Assume that $K_0$ has a subfield $L$ such that $K/L$ is a Galois extension (so in particular, $K/K_0$ is Galois).
Let $\Ga:=Gal(K/L)$.
The group $\Ga$ acts on the set of isomorphism classes of structures of the same type as $W$ of dimension $dim(W)$
(this is just the Galois action on the structure constants).
Denote by $H<\Ga$ the stabilizer of the isomorphism class of $W$. Then $K^H=K_0$.
\end{theorem}

The field $K_0$ also plays a role in studying polynomial identities of $W$. 
In Section \ref{identities} we will explain some connections between the category $\C_W$ and polynomial identities of $W$ 
(for this part only we need $W$ to be an algebra or an $H$-comodule algebra).
We will explain how polynomial identities can be understood in the context of our category $\C_W$, and we will prove the following theorem:
\begin{theorem}\label{main5}
Let $W$ be an algebra or an $H$-comodule algebra of finite dimension over $K$, where $H$ is a finite dimensional Hopf algebra over a subfield $k\subseteq K$. 
Then all the identities of $W$ are already defined over $K_0$ (In case $W$ is an $H$-comodule algebra, we will necessarily have that $k\subseteq K_0$).
\end{theorem}

Finally, we will give some concrete examples in Sections \ref{examples4}-\ref{examples2}.
In Section \ref{examples0} we will consider a three dimensional associative algebra.
We will describe $K_0$ in that case, and we will show that all the identities are already defined over a proper subfield of $K_0$.

In Section \ref{examples1} we will consider the case where $W$ is the algebra $M_n(K)$.
In this case $K_0=\Q$ and we will see that we can choose $B_W$ to be a localization of $K[M_n(K)\times M_n(K)]^{PGL_n}$,
where the action of $PGL_n$ is by conjugation. This construction is not new.
Amitsur was the first to construct a generic division algebra $R$ by using the polynomial identities of $M_n(K)$. 
Later, Procesi gave an alternative description of this algebra as the subalgebra of $M_n(K[x_{ij}^l])$ generated by generic matrices.
Procesi introduced also the algebra $\bar{R}$, which is formed by joining to $R$ the traces of all elements in $R$ (which are central in $R$).
The reader is referred to the paper \cite{Formanek} by Formanek for a survey on this subject.
The generic form we get here will be a localization of $\bar{R}$ by a central element.

In Section \ref{examples2} we will study in detail the case of a twisted Hopf algebra. 
More precisely, let $H$ be a Hopf algebra defined over a subfield $k\subseteq K$, let $\alpha:H\ot H\rightarrow K$ be a convolution invertible two-cocycle, 
and let $W=\, ^{\alpha}H$ be the resulting twisted algebra.
Then $W$ is an $H$-comodule algebra (the structure we will consider for $W$ here will be the multiplication in $W$ and the action of $H^*$ on $W$).
Such $H$-comodule algebras can be thought of as the noncommutative analogue of a principal fiber bundle, where the group $G$ is replaced by the Hopf algebra $H$.
(see e.g. \cite{Schneider}). They also coincide with the class of cleft Hopf Galois extensions of the ground field.
In \cite{Alkass} Aljadeff and Kassel studied such comodule algebras for general Hopf algebras (not necessarily finite dimensional ones) by means of polynomial $H$-identities.
They have constructed an algebra $\A^{\alpha}_H$,
which is formed by taking the most general two-cocycle cohomologous to $\alpha$. 
This algebra can be seen as a Hopf Galois extension of the commutative subalgebra $\B^{\alpha}_H:=(\A^{\alpha}_H)^{coH}$. 
They have shown that if $^{\beta}H$ is a form of $^{\alpha}H$ over an extension $L$ of $k$, then there exists a homomorphism $\phi:\B^{\alpha}_H\rightarrow L$ such that 
$\A^{\alpha}_H\ot_{\B^{\alpha}_H} L\cong\, ^{\beta}H$, 
and that every homomorphism $\B^{\alpha}_H\rightarrow L$ will give rise to a form of $^{\alpha}H$ if a certain integrality condition holds.
In \cite{Kasmas} this integrality condition was proved for finite dimensional Hopf algebras (among other cases). 
Thus, $\A^{\alpha}_H$ is a generic form of $^{\alpha}H$. 
A generic form for a twisted group algebra was constructed in \cite{AHN}. 
The nature of our construction of the generic form is different from that of $\A^{\alpha}_H$.
The main difference is the fact that here we will concentrate more on the algebra structure and the action of $H^*$ than on the fact that this algebra arises from a two-cocycle.
For a twisted group algebra, the construction we will give here will be very similar to the one which appears in \cite{AHN}. 
The main difference is that our base ring will have a smaller Krull dimension.
We will discuss the generic construction for group algebras, Taft algebras and products of Taft algebras.
In \cite{Iyerkas} Iyer and Kassel have studied the base ring $\B^{\alpha}_H=(\A^{\alpha}_H)^{coH}$ in case $\alpha$ is trivial for several families of Hopf algebras, 
including the Taft algebras.
Our construction will give us a base ring of smaller Krull dimension.
We will also show how the construction of the fundamental category can help us to classify all cocycles on $H$, where $H$ is a Taft algebra or a product of Taft algebras.

This paper is organized as follows:
In Section \ref{prel} we give some preliminaries about structures and monoidal categories.
In Section \ref{construction} we will construct the category $\C_W$, based on the construction of kernel completions presented in Section \ref{sec:ker-comp}. 
In Section \ref{properties} we will describe the way the category $\C_W$ behaves with respect to field extensions. 
We will show that all forms of $W$ give rise to equivalent categories. 
In Section \ref{cisform} we will prove Theorem \ref{main1},
and in Section \ref{generic} we will construct the generic form and prove Theorem \ref{main2}.
In Section \ref{sec:galois} we will explain the connection of our construction to classical descent theory and we will prove Theorem \ref{main4}.
In Section \ref{identities} we will explain the relation with polynomial identities and prove Theorem \ref{main5}.
Finally, we will give examples in Sections \ref{examples4}- \ref{examples2}.
\end{section}

\begin{section}{Preliminaries}\label{prel}
\begin{subsection}{Monoidal categories}
We will recall here some facts about monoidal categories.
For more detailed discussion the reader is referred to Chapter VII of \cite{maclane}, Chapter XI of \cite{Kassel} 
and to the papers  \cite{Deligne-tensor}, \cite{Deligne-tannaka} and \cite{Deligne-Milne}.
A \textit{monoidal category} $\C$ is a category equipped with a product (which we shall always denote by $\ot$)
$$\ot:\C\times \C\rightarrow \C.$$ The category contains a unit object $\one$ with respect to that multiplication,
and we have, for every $X,Y,Z\in Ob\C$ functorial isomorphisms: $$\lambda_X:\one\ot X\rightarrow X,$$ 
$$\rho_X:X\ot\one\rightarrow X$$ $$\textrm{and } \alpha_{X,Y,Z}:(X\ot Y)\ot Z\rightarrow X\ot (Y\ot Z).$$ 
These functorial isomorphisms should satisfy $\rho_{\one}=\lambda_{\one}$ and should make the following two diagrams commute: \\
$\xymatrix{ & & & (X\ot\one)\ot Y\ar[r]\ar[d] & X\ot(\one\ot Y)\ar[d] \\ & & & X\ot Y\ar[r]^{=} & X\ot Y}$\\
$\xymatrix{ (X\ot (Y\ot Z))\ot W\ar[rr] & & X\ot((Y\ot Z)\ot W)\ar[d] \\
((X\ot Y)\ot Z)\ot W\ar[u]\ar[rd] & & X\ot(Y\ot (Z\ot W)) \\
 & (X\ot Y)\ot (Z\ot W)\ar[ru]}$
We will often omit the functorial isomorphism $\alpha$ in our definitions and computations.
This will do no harm, since by Mac Lane coherence the associativity isomorphisms can be inserted to every diagram in a unique way.
We can thus talk freely about the $n$-th fold tensor product $X^{\ot n}$ of an object $X$ in $\C$.
A monoidal category is called \textit{symmetric} if in addition we have a functorial isomorphism $c_{X,Y}:X\ot Y\rightarrow Y\ot X$ which satisfies:\\
1. $c_{Y,X}c_{X,Y}=Id_{X\ot Y}$ for every two objects $X$ and $Y$.\\
2. $(c_{X,Z}\ot Id_Y)(Id_X\ot c_{Y,Z}) = c_{X\ot Y,Z}$ for every three objects $X$, $Y$ and $Z$.\\
If $X$ is an object of a symmetric monoidal category $\C$, then a dual object of $X$ is an object $X^*$, equipped with two morphisms:
$ev_X:X^*\ot X\rightarrow \one$ and $coev_X:\one\rightarrow X\ot X^*$ such that the following two morphisms are the identity morphisms:
\begin{equation}\label{duality1}X\stackrel{coev_X\ot Id_X}{\rightarrow} X\ot X^*\ot X\stackrel{Id_X\ot ev_X}{\rightarrow} X\end{equation}
$$X^*\stackrel{Id_X\ot coev_X}{\rightarrow} X^*\ot X\ot X^*\stackrel{ev_X\ot Id_X}{\rightarrow} X^*.$$
We say that $\C$ is \textit{rigid} in case every object has a dual.
In this case duality extends in a natural way to a contravariant functor $^*:\C\ra\C$:
if $f:X\ra Y$ is a morphism in $\C$ then the dual morphism is given by the composition $Y^*\ra Y^*\ot X\ot X^*\ra Y^*\ot Y\ot X^*\ra X^*$.
In general monoidal categories we need to be more careful and define right and left duals.
Since the category $\C$ is symmetric, we can avoid this distinction:
by using the symmetry operations the left dual coincides with the right dual, and as a result we have a natural isomorphism $X\cong X^{**}$.
When $\C$ is rigid we have functorial isomorphisms for every $X,Y,Z\in Ob \C$:
$$Hom_{\C}(X^*\ot Y,Z)\cong Hom_{\C}(Y,X\ot Z)$$
$$Hom_{\C}(Y\ot X,Z)\cong Hom_{\C}(Y,Z\ot X^*).$$
Using these isomorphisms one can prove that if $X$ is an object of a rigid abelian symmetric monoidal category $\C$
then the functor $X\ot - : \C\ra \C$ is exact.

If $\C$ and $\D$ are two monoidal categories,
then a functor $F:\C\rightarrow \D$ is called \textit{monoidal} if it is equipped with functorial isomorphisms
$f_{X,Y}:F(X\ot Y)\rightarrow F(X)\ot F(Y)$ which are compatible with the associativity isomorphisms in $\C$ and in $\D$ 
(we refer to XI.4.1 in the book \cite{Kassel} for an exact definition).
If $\C$ and $\D$ are symmetric monoidal categories and $F:\C\rightarrow \D$ is a monoidal functor,
then $F$ is called \textit{symmetric} in case the following diagram commutes:\\
$\xymatrix{& & & F(X\ot Y)\ar[r]\ar[d] & F(Y\ot X)\ar[d] \\ 
	   & & & F(X)\ot F(Y)\ar[r] & F(Y)\ot F(X).}$\\
In other words, $F$ should ``translate'' the symmetry in $\C$ to the symmetry in $\D$.
Notice that if $\C$ and $\D$ are rigid, we will get an isomorphism $F(X^*)\cong F(X)^*$ without further restrictions on $F$.

If $K$ is any field, a \textit{$K$-linear category} is an abelian category in which all the homomorphism groups are $K$-vector spaces 
and in which the composition of morphisms is $K$-bilinear.
Our main focus in this paper will be on categories $\C$ which are $K$-linear symmetric rigid monoidal categories (for some field $K$).
We will further assume that $End_{\C}(\one)=K$.
The general example to keep in mind is the following:
let $G$ be an affine algebraic group over $K$. Then the category $\C=Rep_K-G$ of all rational finite dimensional representations of $G$
is such a category. Indeed, this category is abelian. If $V$ and $W$ are two $G$-representations then $V\ot W$ is also a representation,
by the diagonal action, and $V^*$ is a representation by the dual action: $g\cdot f = f(g^{-1}-)$ for $g\in G$ and $f\in V^*$.
The tensor identity $\one$ is the trivial one dimensional representation.
In fact, all the examples which we will encounter in this paper will be forms of $Rep_K-G$ (we will prove this in Section \ref{categoryform}).

If we take $G$ to be the trivial group we get the category $Vec_K$ of finite dimensional $K$-vector spaces.
An exact additive faithful symmetric monoidal functor $F:\C\rightarrow Vec_K$ (or more generally, $F:\C\rightarrow B-mod$ 
where $B$ is some commutative algebra) is also called a \textit{fiber functor}.
For example, if $\C=Rep_K-G$, then the forgetful functor (which ``forgets'' the action of the group $G$) $F:\C\rightarrow Vec_K$ is a fiber functor.
Fiber functors will play a central role in the sequel. 

If $\C$ is a $K$-linear rigid symmetric monoidal category, we can carry a lot of the constructions usually done in $Vec_K$ in $\C$.
For example, we can still talk about algebras inside $\C$: an algebra will be an object $A\in \C$ together with a morphism $m:A\ot A\rightarrow A$.
The algebra $A$ is said to be associative in case the two morphisms $m(m\ot 1),m(1\ot m):A\ot A\ot A\rightarrow A$ are equal.
It is said to be commutative if $m=m c_{A,A}$.
If $X$ is any object of $\C$, we can still construct the tensor algebra $T(X)$ and its maximal commutative quotient $Sym(X)$.
These algebras are formed as infinte direct sums of objects of $\C$.
Therefore, they will not necessarily be contained in $\C$, but in a bigger category, $Ind(\C)$ (See Section 2.2 of \cite{Deligne-tensor} for an exact definition.
This will not make much difference here).

If $A$ is an associative commutative algebra with unit in $\C$, we can form localizations of $A$ inside $\C$.
In the classical case, where $\C=Vec_K$, we take an element $f\in A$ and add the inverse of $f$.
We can think of $f$ as the morphism $\one\ra A$ which maps the 1-dimensional vector space $K$ to $Kf$.
In our setting, we will consider a morphisms $f:D\ra A$, where $D$ is an invertible object of $\C$ (so that $ev:D\ot D^*\rightarrow \one$ is an isomorphism).
We will think of such a morphism as an element of $A$.
We form the localization $A_f$ in the following way:
let $\ti{A}:=A\ot Sym(D^*)$ (the tensor product of two commutative algebras in $\C$ is again a commutative algebra). 
We have two maps $\one\rightarrow\ti{A}$. The first is the one that sends $\one $ to the identity of $\ti{A}$.
The second one is given as the composition
$$\one\stackrel{coev_D}{\rightarrow} D\ot D^*\stackrel{f\ot 1}{\rightarrow} A\ot D^*\subseteq \ti{A}.$$
We consider the difference of these two maps $\ti{f}:\one\rightarrow \ti{A}$ and we define $A_f:=\ti{A}/(\ti{f})$
where $(\ti{f})$ is the ideal generated by the image of $\ti{f}$.
For example, if $\C=Rep_K-G$, then $D$ is invertible if and only if its dimension as a vector space is one.
The image of $f$ will then be $f(D)=K\cdot \hat{f}$ for some $\hat{f}\in A$, and the localization $A_f$ will be the same as $A_{\hat{f}}$.
We will sometimes write $A_{\hat{f}}$ for $A_f$, where no confusion can arise.

A module $M$ over an associative algebra $A$ will be an object of $\C$ together with a morphism $m_M:A\ot M\rightarrow M$ such that the usual module axioms hold.
If $A$ is an associative commutative algebra in a $K$-linear symmetric monoidal category $\C$, 
then the category $A-mod$ of $A$-modules inside $\C$ is again a $K$-linear symmetric monoidal category.
The tensor product of two $A$-modules is given by $M\ot_A N:= Coker(g: M\ot A\ot N\rightarrow M\ot N)$, 
where $g$ is given by $m_M\ot 1_N - 1_M\ot m_N$ (we use here the fact that left and right modules are the same over a commutative algebra).
The objects $M\ot_A N$ is again an $A$-module. 

During the construction of the fundamental category $\C_W$ we will use a few times the following lemma:
\begin{lemma}\label{functorsdiagram}
Consider the following diagram of categories and functors:
$$\xymatrix{ & \B\ar[d]^G \\ \C\ar[r]^F & \A}$$
where the functor $G$ is faithful and the category $\C$ is small. 
Let $ob(H):ob\C\ra ob\B$ be a function and let $\phi_X:G(ob(H)(X)) \ra F(X)$ be a collection of isomorphisms for every object $X$ of $\C$.
Then there exists at most one functor $H:\C\ra \B$ such that $\phi$ induces an isomorphism of functors $GH\cong F$, and such that $H(X) = ob(H)(X)$ for every object $X$ in $\C$.
The functor $H$ exists if and only if for every two objects $X,Y$ in $\C$, the image of 
$Hom_{\B}(ob(H)(X),ob(H)(Y))\ra Hom_{\A}(G(ob(H)(X)),G(ob(H)(Y)))\cong Hom_{\A}(F(X),F(Y))$ contains the image of the map 
$Hom_{\C}(X,Y))\ra Hom_{\A}(F(X),F(Y))$.
The functor $H$ is then given explicitly on morphisms by $H(f) = G^{-1}(\phi_Y^{-1}F(f)\phi_X)$ where $f:X\ra Y$.
Moreover, if the functor $F$ is faithful, then $H$ is faithful as well.
 \end{lemma}
\begin{proof}
Let $X,Y$ be objects of $\C$. If $H$ exists then the composition 
$Hom_{\C}(X,Y)\stackrel{H}{\ra} Hom_{\B}(H(X),H(Y))\stackrel{G}{\ra} Hom_{\A}(GF(X),GH(Y))\stackrel{\phi}{\ra}Hom_{\A}(F(X),F(Y))$
is equal to $Hom_{\C}(X,Y)\stackrel{F}{\ra} Hom_{\A}(F(X),F(Y))$. 
This shows that the condition in the lemma holds. 
Conversely, Since $G$ is faithful, for every morphism $f:X\ra Y$ in $\C$ there is only one possible way to define
$H(f)$ in a way which will make the diagram commutative.
The morphism $H(f)$ will be the unique morphism whose image under
$Hom_{\B}(H(X),H(Y))\ra Hom_{\A}(GH(X),GH(Y))\ra Hom_{\A}(F(X),F(Y))$ is equal to $F(f)$.
If the condition of the lemma holds there is such a morphism,
and by using again the faithfulness of $G$ we can prove that if we define $H$ in this way on morphisms we will get a functor.
The fact that faithfulness of $F$ implies faithfulness of $H$ is immediate. 
\end{proof}
\end{subsection}
\begin{subsection}{Structures}
Let $W$ be an object of a symmetric monoidal rigid category $\C$.
Let $x_i\in Hom_{\C}(\one,W^{p_i,q_i})$ where $W^{p_i,q_i}:=W^{\ot p_i}\ot (W^*)^{\ot q_i}$ be a collection of morphisms
(which we shall also call \textit{tensors}).
We call the pair $(W,\{x_i\})$ an \textit{algebraic structure} in $\C$ or just a structure in $\C$.
Usually we will just refer to $W$ as a structure, without mentioning the tensoors $\{x_i\}$.
For example, if $W$ is an algebra in $Vec_K$ then the algebra structure is given by the multiplication map $m: W\ot W\rightarrow W$,
which can also be regarded as a map $m\in Hom_{K}(\one,W^{1,2})$ (by using the isomorphisms in Equation \ref{duality1}
we have that $Hom_{\C}(V,U)\cong Hom_{\C}(\one,V^*\ot U)$ for every $U,V\in ob \C$.
In case $\C=Vec_K$ we also have an identification $Hom_{\C}(V,U)\cong V^*\ot U)$).
If $W$ is a module over an algebra $A$ (in $Vec_K$), 
then for every $a\in A$ we have a tensor $T_a\in W^{1,1}$ that specifies the action of $a$. 
In particular, if $A=K[x]$, then an $A$-module will be the same as a vector space $W$ together with a single endomorphism $T:W\ra W$.
If $H$ is a finite dimensional Hopf algebra over $K$ and $W$ is an $H$-comodule algebra, 
then the structure of $W$ contains the multiplication $m\in W^{1,2}$ together with linear endomorphisms $T_f\in W^{1,1}$, one for each $f\in H^*$.
More details on this structure will be given in Section \ref{examples2}.

If $x\in W^{p,q}$ we call $(p,q)$ the \textit{type} of the tensor $x$.
The type of a tensor is thus an element of $\mathbb{N}^2$ (we consider here 0 as a natural number).
Let $(W,\{x_i\})$ and $(W',\{y_i\})$ be two structures such that the type $(p_i.q_i)$ of $x_i$ is the same as the type of $y_i$ for each $i$.
If $\psi:W\rightarrow W'$ is an isomorphism in $\C$, then $\psi$ induces an isomorphism 
$\psi^{p,q}=\psi^{\ot p}\ot ((\psi^*)^{-1})^{\ot q}:W^{p,q}\rightarrow W'^{p,q}$ for every $(p,q)\in \mathbb{N}^2$
(by using the monoidal and rigid structure of $\C$).
An \textit{isomorphism} between the structures $(W,\{x_i\})$ and $(W',\{y_i\})$ is an isomorphism $\psi:W\rightarrow W'$ such that
$\psi^{p_i,q_i}_*(x_i)=y_i$ for every $i$. 
if $F:\C\ra \D$ is a symmetric monoidal functor between symmetric monoidal rigid categories, then if $(W,\{x_i\})$ is a structure in $\C$,
one can use the monoidal structure on $F$ to get a structure in $\D$ of the same type.
Indeed, we have a map $Hom_{\C}(\one,W^{p,q})\ra Hom_{\D}(F(\one),F(W^{p,q}))\cong 
Hom_{\D}(\one,F(W)^{p,q})$ for every $(p,q)\in\mathbb{N}^2$. 
If we denote the image of $x_i$ under the suitable map by $y_i$, then we get a structure $(F(W),\{y_i\})$ in $\D$.
We will say that the structure $(W,\{x_i\})$ \textit{lies above} the structure $(F(W),\{y_i\})$,
or that $(W,\{x_i\})$ is a \textit{lift} of the structure $(F(W),\{y_i\})$.

Two important instances of this are the following: 
if $L$ is an extension field of $K$, then we have a natural extension of scalars functor $F:Vec_K\ra Vec_L$
given by $F(W) = W\ot_K L$. In this case we say that $W$ is a \textit{form} of $F(W)$.
If $(W,\{x_i\})$ is a structure in $Vec_K$ equipped with the action of an algebraic group $G$ such that $g(x_i) = x_i$ for every $g\in G$ and every $i$,
then $(W,\{x_i\})$ can also be considered as a structure in $\C=Rep_K-G$. The forgetful functor $F:\C\ra Vec_K$
then sends the structure $(W,\{x_i\})$ into itself.
\end{subsection}
\end{section}

\begin{section}{The kernel completion of a category}\label{sec:ker-comp}
Let $\C$ be an additive symmetric monoidal rigid category, and let $F:\C\ra \A$	
be a faithful symmetric monoidal functor from $\C$ 
to an abelian symmetric monoidal rigid category $\A$.
In this section we shall construct a new category, called the kernel completion of $\C$, denoted by $Ker(\C)_F$,
which will give us a first approximation for kernels of morphisms in $\C$.
We shall use this construction in the next section inductively in order to construct the fundamental category $\C_W$.
We begin with a definition:
A sequence $A\ra B\ra C\cdots$ of objects and morphisms in $\C$ is called \textit{good} with respect to $F$
if after applying $F$ to it we get an exact sequence in $\A$.
We say that the category $\C$ is good with respect to $F$ if every morphism $f:A\ra B$ in $\C$
can be embedded in a good sequence of the form: $X\stackrel{g}{\ra} A\stackrel{f}{\ra} B$.
In this case we can also use the condition again for the morphism $g$, and embed $f$ in a good sequence of the form
$Y\stackrel{h}{\ra} X\stackrel{g}{\ra} A\stackrel{f}{\ra} B$.
When the functor $F$ is clear from the context we will just say that $\C$ is good. 
We assume from now on that $\C$ is good and small.
In the next section we shall use linear algebra to explain why this holds for the categories we are interested in.
We claim now the following:
\begin{theorem}\label{ker-comp}
Let $\C$, $F$ and $\A$ be as above. 
There exists an additive symmetric monoidal rigid category $Ker(\C)_F$, called the kernel completion of $\C$ with respect to $F$, 
and symmetric monoidal faithful functors $I:\C\ra Ker(\C)_F$ and $\ti{F}:Ker(\C)_F\ra \A$ such that the following holds:\\
1. The functor $F$ is isomorphic with $\ti{F}I$.\\
2. For every morphism $f:A\ra B$ in $\C$ there exists an object $K(f)$ and a morphism $i_f:K(f)\ra I(A)$ such that the sequence 
$$\ti{F}(K(f))\ra \ti{F}(I(A))\ra \ti{F}(I(B))$$ is exact in $\A$, and such that if $g:C\ra A$ is a morphism in $\C$ which satisfies $fg=0$
then there exists a unique morphism $g':I(C)\ra K(f)$ in $Ker(\C)_F$ such that $i_fg'=I(g)$.\\
3, If $\B$ is an abelian symmetric monoidal rigid category equipped with faithful additive symmetric monoidal functors $J:\C\ra \B$ and $G:\B\ra \A$ 
such that $G$ is also exact and such that 
$GJ\cong F$ then there exists a unique (up to isomorphism) faithful symmetric monoidal functor $H: Ker(\C)_F\ra \B$ 
such that $HI\cong J$ and $GH\cong \ti{F}$.
\end{theorem}
The rest of this section will be used to define the category $Ker(\C)_F$ and prove Theorem \ref{ker-comp}.
We start with constructing the hom-sets of $Ker(\C)_F$.
Since the category $\C$ is small we can speak about the set of all morphisms in $\C$.
We give the following definition:
\begin{definition}\label{def:homsets}
The collection $X^{f,g}\subseteq Hom_{\A}(Ker(F(f)),Ker(F(g))) $ where $f$ and $g$ are morphisms in $\C$ 
is the smallest collection of abelian subgroups which satisfy the following conditions:\\
1. If $X\stackrel{f_1}{\ra}Y\stackrel{f_0}{\ra} A \stackrel{f}{\ra}B$ is a good sequence in $\C$ and 
$g:C\ra D$ and $h:Y\ra C$ are morphisms in $\C$ which satisfies $gh=0$ and $hf_1=0$, then the induced map 
$\wb{F(h)}:Ker(F(f))\cong Coker(F(f_1))\ra Ker(F(g))$
belongs to $X^{f,g}$.\\
2. The collection $X^{f,g}$ is closed under composition: if $a\in X^{f,g}$ and $b\in X^{g,h}$ then $ba\in X^{f,h}$.\\
3. If $a\in X^{f,g}$ is invertible in $\A$, then $a^{-1}\in X^{g,f}$.
\end{definition}
It is clear that the collection $X^{f,g}$ exists and it is clear how to construct it: 
we start with all the morphisms arising from Condition 1, and we perform closure operation to fulfil Conditions 2 and 3.
We next prove the following:
\begin{lemma}
For every morphism $f$ in $\C$, we have $Id_{Ker(F(f))}\in X^{f,f}$. 
\end{lemma}
\begin{proof}
Let $f:A\ra B$ and let $X\stackrel{f_1}{\ra}Y\stackrel{f_0}{\ra} A \stackrel{f}{\ra}B$ 
be a good sequence in $\C$. 
The map $\wb{F(f_0)}:Ker(F(f))\ra Ker(F(f))$ is then the identity map $I_{Ker(F(f))}$
\end{proof}

\begin{definition}
We define the category $Ker(\C)_F$ as follows:
the objects of the category are in one-to-one correspondence with the morphisms in $\C$ and are denoted by $K(f)$.
The morphisms are $$Hom_{Ker(\C)_F}(K(f),K(g)) = X^{f,g}.$$
The composition of morphisms is the same as composition of morphisms in $\A$.
\end{definition}
By the construction of $X^{f,g}$ it is clear that the composition is well defined,
and by the previous lemma we see that $Hom_{Ker(\C)_F}(K(f),K(f))$ contains the identity morphism.
We next claim the following lemma:
\begin{lemma}
The category $Ker(\C)_F$ is an additive category. 
\end{lemma}
\begin{proof}
All the hom-sets in $Ker(\C)_F$ are abelian groups.
If $f:A\ra B$ and $g:C\ra D$ are morphisms in $\C$, then we define $K(f)\oplus K(g)$ to be $K(f\oplus g)$
where $f\oplus g:A\oplus C\ra B\oplus D$. We use two good sequences $X_1\ra Y_1\ra A\ra B$ and $X_2\ra Y_2\ra C\ra D$ 
and their direct sum (which is also a good sequence) 
in order to prove that the canonical injections and projections are contained in $Hom_{Ker(\C)_F}(K(f\oplus g),K(f))$,
$Hom_{Ker(\C)_F}(K(f),K(f\oplus g))$ and similarly for $g$.
\end{proof}

We define $\ti{F}(K(f)) = Ker(F(f))$ and we define $\ti{F}$ on morphisms to be the inclusion
$Hom_{Ker(\C)_F}(K(f),K(g))\subseteq Hom_{\A}(Ker(F(f)),Ker(F(g)))$.
Then it is easy to see that $\ti{F}$ is an additive faithful functor.
We next define $I:\C\ra Ker(\C)_F$ in the following way:
for every object $A$ of $\C$ we have the unique morphism $0_A:A\ra 0$.
We define $I(A)= K(0_A)$.
If $f:A\ra B$ is a morphism in $\C$, then we use the diagram
$$\xymatrix{0\ar[r]\ar[d] & A\ar[r]^{1_A}\ar[d]^f & A\ar[r]\ar[d]^f &  0 \\ 0 \ar[r] & B\ar[r]^{1_B} & B\ar[r] & 0 }$$
whose rows are good to show that $F(f)\in X^{0_A,0_B}$. We define $I(f) = F(f)\in Hom_{Ker(\C)_F}(K(0_A),K(0_B))$.
Again, a direct verification shows that $I$ is also faithful and additive, and that $\ti{F}I\cong F$.

By using a good sequence $X\stackrel{f_1}{\ra}Y\stackrel{f_0}{\ra} A \stackrel{f}{\ra}B$ 
we get, by using Condition 1 in Definition \ref{def:homsets}, that the canonical inclusion $i_f:K(f)\ra I(A)$ is contained in $Hom_{Ker(\C)_F}(K(F),I(A))$.
By using Condition 1 again we get the desired map of Part 2 of Theorem \ref{ker-comp}.
We have thus proved Parts 1 and 2 of Theorem \ref{ker-comp}.
We next prove that $Ker(\C)_F$ is indeed a symmetric monoidal rigid category,
that all the functors are symmetric monoidal functors, and that the universal property from Part 3 of Theorem \ref{ker-comp} holds.

\begin{lemma}
 The category $Ker(\C)_F$ is a symmetric monoidal rigid category. The functors $\ti{F}$ and $I$ constructed above are symmetric monoidal functors.
\end{lemma}
\begin{proof}
Let $f:A_1\ra B_1$ and $g:A_2\ra B_2$ be two morphisms in $\C$. We define
$$K(f)\ot K(g):= K(f\ot 1_{A_2}\oplus 1_{A_1}\ot g)$$
where $f\ot 1_{A_2}\oplus 1_{A_1}\ot g: A_1\ot A_2\ra A_1\ot B_2\oplus B_1\ot A_2$.
In order to extend the definition of the tensor product on objects to morphisms, we use Lemma \ref{functorsdiagram}.
We have the following diagram of categories and functors:
$$\xymatrix{ & & Ker(\C)_F\ar[d]^{\ti{F}} \\ Ker(\C)_F\times Ker(\C)_F\ar[rr]^{\otimes(\ti{F},\ti{F})} & &\A}$$
Since tensor product with a given object is exact in $\A$, we have a natural isomorphism $\ti{F}(K(f\ot 1_{A_2}\oplus 1_{A_1}\ot g))\cong Ker(F(f))\ot Ker(F(g))$. 
We are in the situation of Lemma \ref{functorsdiagram}.
We need to show that the condition of the lemma holds here.
For this, we need to show that the image of every morphism in $Ker(\C)_F\times Ker(\C)_F$ under $\ot(\ti{F},\ti{F})$ is contained in the image of the functor
$\ti{F}:Ker(\C_F)\ra \A$. We will call such a morphism \textit{liftable} (we shall use this notation again in every application of Lemma \ref{functorsdiagram}).
For this, We shall use again Definition \ref{def:homsets}.
Since the collection of liftable morphisms is closed under composition, it is enough to prove that all morphisms
of the form $(a,1):(K(f),K(h))\ra (K(g),K(h))$ where $a:K(f)\ra K(g)$ is a morphism in $Ker(\C)_F$ are liftable.

We start with the morphisms which appear in Condition 1 of Definition \ref{def:homsets}.
Let then $f:A_1\ra B_1$, $g:A_2\ra B_2$ $h:A_3\ra B_3$ be three morphisms in $\C$. 
Let $l:K(f)\ra K(g)$ be a morphism which arises from a good sequence $X_1\stackrel{f_1}{\ra}Y_1\stackrel{f_0}{\ra} A_1 \stackrel{f}{\ra}B_1$ and a morphism 
$\hat{l}:Y_1\ra A_2$ such that $g\hat{l}=0$ and $\hat{l}f_1=0$.
We would like to show that the morphism $(l,1):(K(f),K(h))\ra (K(g),K(h))$ is liftable.
For this, we embed $h$ in a good sequence $X_3\stackrel{h_1}{\ra}Y_3\stackrel{h_0}{\ra} A_3 \stackrel{h}{\ra}B_3$.
We then have the following good sequence:
$$ X_1\ot Y_3\oplus Y_1\ot X_3\ra Y_1\ot Y_3\ra A_1\ot A_3\ra A_1\ot B_3\oplus B_1\ot A_3$$
The morphism $\hat{l}\ot h_0: Y_1\ot Y_3\ra A_2\ot A_3$ then induces the desired morphism $K(f)\ot K(h)\ra K(g)\ot K(h)$.
The collection of liftable morphisms is closed under composition, so Condition 2 is satisfied trivially.
For Condition 3, notice that if $l:K(f)\ra K(g)$ is liftable and becomes an isomorphism after applying $F$, then $l\ot 1: K(f)\ot K(h)\ra K(g)\ot K(h)$ 
also becomes an isomorphism after applying $F$.
Since a morphism in $Ker(\C)_F$ which becomes invertible after applying $F$ is invertible,
we get that $(l^{-1},1)$ is liftable as well. This finishes the proof that the category $Ker(\C)_F$ is monoidal.
We still need to prove that $Ker(\C)_F$ is also symmetric. 
For this we just write down the isomorphism $K(f)\ot K(g)\ra K(g)\ot K(f)$ 
which is induced from the symmetry isomorphism $A_1\ot A_2\ra A_2\ot A_1$ in $\C$.

Finally, we prove that $Ker(\C)_F$ is rigid. 
We need to prove that for every morphism $f:A\ra B$ the object $K(f)$ admits a dual.
For this, we will use the fact that $\A$ is rigid, and that $\C$ is good.
Since the duality functor is exact, we have a natural isomorphism $Ker(F(f))^*\cong Coker(F(f)^*)$ in $\A$.
We embed $f$ in a good sequence $X\stackrel{f_1}{\ra}Y\stackrel{f_0}{\ra} A\stackrel{f}{\ra}B$.
The dual sequence $B^*\stackrel{f^*}{\ra} A^*\stackrel{f_0^*}{\ra} Y^*\stackrel{f_1^*}{\ra} X^*$ is also good,
and we have an isomorphism $Ker(F(f))^*\cong Coker(F(f)^*)\cong Ker(F(f_1)^*)\cong Ker(F(f_1^*))$.
So the object $K(f_1^*)$ of $Ker(\C)_F$ is mapped under $\ti{F}$ to an object of $\A$ which is isomorphic with the dual of $\ti{F}(K(f))$.
We have that $K(f_1^*)\ot K(f) = K(f_1^*\ot 1\oplus 1\ot f)$, and the following good sequence: 
$B^*\ot Y\oplus A^*\ot X\ra A^*\ot Y\ra Y^*\ot A\ra Y^*\ot B\oplus X^*\ot A$.
The morphism $A^*\ot Y\ra A^*\ot A\ra \one$ in $\C$ induces then a morphism $\mu:K(f_1^*)\ot K(f)= K(f_1^*\ot 1\oplus 1\ot f)\ra I(\one)=\one$
in $Ker(\C)_F$. In a similar way we get a morphism $\nu:\one\ra K(f)\ot K(f_1^*)$.
If we consider the isomorphism $\ti{F}(K(f_1^*))\cong Ker(F(f))^*$
we see that these $\ti{F}(\mu)$ and $\ti{F}(\nu)$ are exactly the evaluation and coevaluation morphisms for $Ker(F(f))$.
Thus, the composition $\lambda:K(f)\ra K(f)\ot K(f_1^*)\ot K(f)\ra K(f)$ maps under $\ti{F}$ to the identity morphism.  
Since the functor $\ti{F}$ is faithful, we get that $\lambda = Id_{K(f)}$. 
In a similar way, one can show that the composition $K(f_1^*)\ra K(f_1)^*\ot K(f)\ot K(f_1)^*\ra K(f_1^*)$ is the identity morphisms.
This implies that $K(f_1^*)$ is the dual of $K(f)$ in $Ker(\C)_F$ as desired. 
This finishes the proof that $Ker(\C)_F$ is a symmetric monoidal rigid category.
To prove that the functors $I$ and $\ti{F}$ are symmetric monoidal functors, we notice that 
if $A$ and $B$ are objects of $\C$, then we have a natural isomorphism $I(A)\ot I(B) = K(0_A)\ot K(0_B)=K(0_A\ot 0_B)\cong K(0_{A\ot B})\cong I(A\ot B)$ which
makes $I$ into a monoidal functor, and if $f:A_1\ra B_1$ and $g:A_2\ra B_2$ are two morphisms in $\C$, then we have
$\ti{F}(K(f)\ot K(g)) = \ti{F}(K(f\ot 1_{A_2}\oplus 1_{A_1}\ot g)) = Ker(F(f\ot 1_{A_2}\oplus 1_{A_1}\ot g))\cong Ker(F(f)\ot 1_{F(A_2)}\oplus 1_{F(A_1)}\ot F(g))\cong
Ker(F(f))\ot Ker(F(g)) = \ti{F}(K(f))\ot \ti{F}(K(g))$, which shows that $\ti{F}$ is also a monoidal functor.
The proof that the functors $I$ and $\ti{F}$ are also symmetric follows from using the symmetry morphisms in $\C$ and in $\A$, 
and using the fact that the functor $F$ is symmetric.
\end{proof}
Finally, we prove that the category $Ker(\C)_F$ satisfies the universal property of Part 3 of Theorem \ref{ker-comp}.
Assume then that we have $\C\stackrel{J}{\ra}\B\stackrel{G}{\ra}\A$ such that $GJ\cong F$ as in the statement of the theorem.
We define $H:Ker(\C)_F\ra \B$ in the following way:
on objects $H(K(f)) = Ker(J(f))$. 
We then have $G(H(K(f))) = G(Ker(J(f)) \cong Ker(GJ(f)) \cong Ker(F(f))= F(K(f))$
because $G$ is exact. 
To define $H$ on morphisms we use Lemma \ref{functorsdiagram}.
We have a diagram of the form
$$\xymatrix{ & \B\ar[d]^G\\ Ker(\C)_F\ar[r]^{\ti{F}} & \A}$$
where the functor $G$ is faithful, and where we already defined $H$ on objects of $Ker(\C)_F$.
To check that the conditions of the lemma are satisfied, we will use Definition \ref{def:homsets}.
We need to show that morphisms which appear in Condition 1 are liftable,
and that the collection of liftable morphisms is closed under the operations in Conditions 2 and 3.
The fact that the collection of liftable morphisms is closed under composition is immediate,
so we have Condition 2.
If $f$ is a morphism in $\B$ such that $G(f)$ is invertible in $\A$, then $G(Ker(f)) = Ker(G(f)) = 0$ and $G(Coker(f)) = Coker(G(f))=0$ since $G$ is exact. We conclude that
$Ker(f) =0$ and $Coker(f)=0$ since $G$ is faithful. Therefore, $f$ is already invertible in $\B$ (again, because $\B$ is abelian), and so we have Condition 3.
For morphisms which appear in Condition 1 we use the fact that $\B$ is an abelian  category.
Let $X\stackrel{f_1}{\ra}Y\stackrel{f_0}{\ra} A \stackrel{f}{\ra}B$ be a good sequence in $\C$, and let
$g:C\ra D$ and $h:Y\ra C$ be morphisms in $\C$ such that $gh=0$ and $hf_1=0$.
We thus get an induced morphism $\ti{h}:K(f)\ra K(g)$ which maps under $\ti{F}$ to the morphism $\wb{F(h)}:Ker(F(f))\cong Coker(F(f_1))\ra Ker(F(g))$. 
Since $\B$ is abelian and $J$ is additive, the sequence $J(X)\stackrel{J(f_1)}{\ra}J(Y)\stackrel{J(f_0)}{\ra} J(A) \stackrel{J(f)}{\ra}J(B)$ is still a complex, 
and we have an induced morphism $Coker(J(f_1))\ra Ker(J(f))$.
This morphim becomes invertible after applying $G$, and so it is already invertible in $\B$. 
The map $Ker(J(f))\cong Coker(J(f_1))\ra Ker(J(g))$ in $\B$ is therefore mapped under $G$ to the map $\wb{F(h)}$
and so the map $\ti{h}$ is liftable, as desired
(we use here the fact that $G$ is exact, and so it maps the kernel of $J(g)$ to the kernel of $F(g)$). 
This shows that the conditions of Lemma \ref{functorsdiagram} are fulfilled, and we therefore have our desired functor $H:Ker(\C)_F\ra \B$.
Checking that $J\cong HI$ and $GH\cong \ti{F}$ is straightforward. The proof that the functor $H$ is symmetric monoidal follows
the lines of the proof that the functors $I$ and $\ti{F}$ are symmetric monoidal.
This finishes the proof of Theorem \ref{ker-comp}.
We remark here that since $Ker(\C)_F$ is a rigid category, and since $Ker(f^*)^*$ is the cokernel of $f$,
it is also true that every morphism in $\C$ has a cokernel in $Ker(\C)_F$.
We will use this property in the next section.
\end{section}

\begin{section}{Construction of the fundamental category and a proof of Theorem \ref{universalproperty}}\label{construction}
Let $K$ be a field of characteristic zero and let $(W,\{x_i\})$ be a finite dimensional structure over $K$.
We are going to construct the fundamental category $\C_W$ of $W$.
We will construct this category as a direct limit in the category of small additive symmetric monoidal rigid categories: $$\C_W=\varinjlim_n\C_n.$$
All the morphisms in $\C_n$ will have kernels and cokernels in $\C_{n+1}$
(this will ensure us that $\C_W$ will be abelian).
Moreover, for each $n$ we will have an additive faithful symmetric monoidal functor $F_n:\C_n\rightarrow Vec_K$,
and their limit will give us a functor $F:\C_W\rightarrow Vec_K$.

\begin{subsection}{Construction of the zeroth and the first category}
We will begin by constructing a pre-additive category, $\C_0$ together with an additive faithful symmetric monoidal functor $F_0:\C_0\rightarrow Vec_K$ 
(by a pre-additive category we mean here a category in which all the hom-sets are abelian groups,
and such that composition of morphsims is bilinear. 
An additive functor is a functor for which the induced map on the hom-sets is a homomorphism of abelian groups).
We then define $\C_1$ to be the additive envelope of $\C_0$.
(we will give later a definition of the additive envelope). We will show that $F_0$ extends to an additive faithful monoidal functor $F_1:\C_1\rightarrow Vec_K$.
The idea is that $\C_0$ will be the tensor category ``freely generated'' by $W$ and $W^*$, 
and the morphisms will be all the morphisms which arise from the tensors (or the ``structure'') $\{x_i\}$.
We begin with defining a collection of $\Q$-subspaces $X^{p,q}\subseteq W^{p,q}$ for every pair $(p,q)\in \mathbb{N}^2$. These subspaces will not necessarily be $K$-subspaces.
However, we will see that $\Q\subset X^{0,0}\subseteq K$ might be a proper intermediate field, and that $X^{p,q}$ will be a vector space over $X^{0,0}$ for every $(p,q)$.

\begin{definition}
 $X^{p,q}\subseteq W^{p,q}$ is the smallest collection of $\Q$-subspaces which satisfies the following conditions:\\
1. $x_i\in X^{p_i,q_i}$\\
2. The identity map $Id_W\in End_K(W)\cong W\otimes W^*$ is contained in $X^{1,1}$.\\
3. If $x\in X^{p,q}$ and $y\in X^{r,s}$ then $x\ot y\in X^{p+r,q+s}$ (after rearranging the factors).\\
4. If $ev_{p,q}:W^{p,q}\rightarrow W^{p-1,q-1}$ is the map which evaluates the first copy of $W^*$ on the first copy of $W$ then $ev(X^{p,q})\subseteq X^{p-1,q-1}$.\\
5. For any $\sigma\in S_p$ and $\tau\in S_q$ we have that $(\sigma,\tau)(X^{p,q})=X^{p,q}$, where the action is given by permuting the tensor factors.\\
6. If $0\neq x\in X^{0,0}\subseteq W^{0,0}=K$, then also $x^{-1}\in X^{0,0}$ (the inversion is made in $K$).
\end{definition}
As in the construction of $Ker(\C)_F$, it is clear that this collection exists, and it is clear how to construct it. 
We just start with the $\Q$-vector spaces generated by the $x_i$'s, and perform some closure operations.
Notice also that by Condition 6, $X^{0,0}$ is a subfield of $K$.

We now construct a pre-additive category $\C_0$ in the following way: the object set of our category will be $\mathbb{N}^2=\{(p,q)\}$.
The morphism groups will be $Hom_{\C_0}((p,q),(a,b))=X^{a+q,b+p}$. Composition is defined in the following way: we have
$X^{a+q,b+p}\subseteq W^{a+q,b+p}=Hom_K(W^{p,q},W^{a,b})$, so
for two morphisms $f:(p,q)\rightarrow (a,b)$ and $g:(a,b)\rightarrow (c,d)$ we can form the composition $gf:W^{p,q}\rightarrow W^{c,d}$ of maps of vector spaces.
By Conditions 3,4 and 5 we have that $gf\in X^{p+d,q+c}$, so this is well defined 
(we use here the fact that when we identify $Hom_K(U,V)$ with $V\ot U^*$ and $Hom_K(V,W)$ with $W\ot V^*$, the composition $(W\ot V^*)\ot (V\ot U^*)\ra W\ot U^*$ 
is given by the evaluation on $V^*\ot V$). 
Notice that by Condition 2 we have the identity maps, and that all the morphism groups are in fact vector spaces over $X^{0,0}$.
The category $\C_0$ is a symmetric rigid monoidal category, and we have a faithful symmetric monoidal additive functor $F_0:\C_0\rightarrow Vec_K$ given by
$F_0((p,q))= W^{p,q}$. 
Indeed, the tensor product of $(p,q)$ with $(a,b)$ will be $(p+a,q+b)$, and the dual of $(p,q)$ will be $(q,p)$. 
From the way we have defined the composition of morphisms it is clear that $F_0$ is really a functor.
The faithfulness of $F_0$ follows from the fact that we defined $Hom_{\C_0}((p,q),(a,b))$ as a subset of $Hom_K(W^{p,q},W^{a,b})$, 
so the induced map on the hom sets is indeed injective.

If $\D$ is any pre-additive category, we can form an additive category, $add(\D)$ (the \textit{additive envelope of $\D$}), 
by simply adding finite direct sums to $\D$ (objects of $add(\D)$
are $n$-tuples of objects of $\D$ (for some natural $n$), and morphisms are given by matrices of morphisms. 
See also the exercises in Section 6.2 of \cite{maclane}).
We have a natural faithful additive functor $I:\D\ra add(\D)$ given by sending each object to itself and each morphism to itself. 
We have the following universal property: if $\A$ is an additive category and $F:\D\ra \A$ is an additive functor 
then there exists a unique additive functor $add(F):add(\D)\ra\A$ such that $add(F)I\cong F$. 
The functor $add(F)$ is given by $add(F)(\oplus A_i) = \oplus F(A_i)$.
If $F$ is faithful then $add(F)$ is faithful as well. 

We define $\C_1=add(\C_0)$. Since $\C_0$ is a small category, $\C_1$ is still a small category (since we add only finite direct sums).
We thus have a natural inclusion functor $I_0:\C_0\ra \C_1$.
The category $\C_1$ is also a symmetric rigid monoidal category over $Vec_{X^{0,0}}$,
and the functor $I_0$ is a faithful symmetric monoidal additive functor. 
Moreover, the functor $F_0$ induces a functor $F_1=add(F_0):\C_1\ra Vec_K$. 
which is also a faithful symmetric monoidal additive functor, 
and we have a natural isomorphism of functors $F_0\cong F_1I_0$.
\end{subsection}

\begin{subsection}{Iterative kernel completions}
We will now use the construction of the previous section in order to add kernels and cokernels to $\C_1$.
In order to do so, we need to overcome a certain difficulty: if $f:A\rightarrow B$ is a morphism in $\C_1$,
then it is clear what should $Hom_{\C_1}(X,Ker(f))$ and $Hom_{\C_1}(Coker(f),Y)$ be,
just by their universal properties. 
However, it is not so clear what should $Hom_{\C_1}(Ker(f),X)$ and $Hom_{\C_1}(Y,Coker(f))$ be.
In order to overcome this obstruction, we will show that our category $\C_1$ is good in the sense of Section \ref{sec:ker-comp}. 
To do so, we will use the symmetric monoidal structure of the category.
We begin with proving the following linear algebra lemma:
\begin{lemma}
 Let $T:U\rightarrow V$ be a linear map between two finite dimensional vector spaces over $K$. Let $k$ be a positive integer.
 Then the image of the map: $$K_T:(V^*)^{\ot k}\ot U^{\ot k+1}\rightarrow U$$ $$f_1\ot f_2\ot\cdots \ot f_k\ot v_1\ot v_2\ot\cdots\ot v_{k+1}\mapsto$$
 $$ \sum_{\sigma\in S_{k+1}}(-1)^{\sigma}f_1(T(v_{\sigma(1)}))f_2(T(v_{\sigma(2)}))\cdots f_k(T(v_{\sigma(k)})) v_{\sigma(k+1)}$$ is equal to:\\
 1. $U$, in case $k<rank(T)$.\\
 2. $Ker(T)$ in case $k=rank(T)$ and \\
 3. 0 in case $k>rank(T)$.\\
\end{lemma}
\begin{proof}
We will concentrate on case number 2. The other cases are easy to deduce.
Assume that $k=rank(T)$.
Let us write $$U=span_K\{ u_1,u_2,\ldots, u_k,u_{k+1},\ldots, u_n\}$$ where $u_{k+1},\ldots, u_n$ span the kernel of $T$,and
$T(u_1),T(u_2),\ldots,T(u_k)$ are linearly independent in $V$.
Let $x = f_1\ot f_2\ot\cdots \ot f_k\ot u_{i_1}\ot u_{i_2}\ot\cdots\ot u_{i_{k+1}}$.
It is easy to see that $K_T(x)=0$ if $|\{i_1,\ldots,i_{k+1}\}|<k+1$.
Now, if two (or more) of the indices $\{i_j\}$ are bigger than $k$ then $K_T(x)=0$ again, because every element in the sum will be zero.
So the only possible way in which $K_T(x)\neq 0$, is if we have that $i_j>k$ for exactly one $j$.
Without loss of generality, let us assume that this is $j=k+1$.
Then we have that:
$$K_T(x) = \sum_{\sigma\in S_{k+1}}((-1)^{\sigma}f_1(T(u_{i_{\sigma(1)}}))f_2(T(u_{i_{\sigma(2)}}))\cdots f_k(T(u_{i_{\sigma(k)}})) u_{i_{\sigma(k+1)}} = $$
$$\sum_{\sigma\in S_{k}}(-1)^{\sigma}f_1(T(u_{i_{\sigma(1)}}))f_2(T(u_{i_{\sigma(2)}}))\cdots f_k(T(u_{i_{\sigma(k)}})) u_{i_{k+1}}\in Ker(T)$$
And by taking the right element from $(V^*)^{\ot k}$ we get all the elements of the kernel of $T$ this way.
Since every element in $(V^*)^{\ot k}\ot U^{\ot k+1}$ is a linear combination of elements of the form of $x$, we are done.
\end{proof}
\begin{remark}
Notice that a dual proof will reveal the fact that the image of $f$ in $V$ is also the kernel of some other morphism with source $V$.
Everything that we will prove in the sequel for kernel of morphisms can be dualize for cokernels.\end{remark}

An important consequence of the previous lemma is the fact that the construction of the map $K_T$
can be done inside our category $\C_1$ (or inside any additive symmetric monoidal rigid category $\C$ 
equipped with an additive faithful symmetric monoidal functor $F:\C\ra Vec_K$).
The reason for this is the following: if $T:U\ra V$ is a morphism in $\C$,
then $C:=(V^*)^{\ot k}\ot U^{\ot k+1}$ where $k=rank(F(f))$ is an object in $\C$.
The morphism $K_T$ is constructed out of the original map $T$ and the action of the symmetric group
(which exists in $\C$ since $\C$ is symmetric monoidal).
It can therefore be constructed in $\C$.
We record this fact in the following corollary:
\begin{corollary}\label{exactness}
Let $\C$ be an additive symmetric monoidal rigid category, and let $F:\C\ra Vec_K$ be a faithful
additive symmetric monoidal functor. 
Then the category $\C$ is good with respect to the functor $F$: any morphism $A\ra B$ in $\C$
can be embedded in a sequence $Y\ra A\ra B$ which becomes exact in $Vec_K$
after applying $F$.
\end{corollary}

In particular, the category $\C_1$ is good with respect to the functor $F_1$. 
We form the kernel-completion of $\C_1$, and denote it by $\C_2$.
So $\C_2$ is also an additive symmetric monoidal rigid category, and we have additive faithful symmetric monoidal functors 
$I_1:\C_1\ra \C_2$ and $F_2:\C_2\ra Vec_K$ such that $F_2I_1\cong F_1$.
By the above corollary, $\C_2$ is good with respect to the functor $F_2$.
We define inductively $\C_{n+1}$ as the kernel completion of $\C_n$ (again, since $\C_n$ is an additive symmetric monoidal category
and the functor $F_n:\C_n\ra Vec_K$ is an additive faithful symmetric monoidal functor, $\C_n$ is good with respect to $F_n$ by the above corollary,
and therefore we can apply Theorem \ref{ker-comp}).
We then have a sequence of additive symmetric monoidal rigid categories $\C_n$ and additive faithful symmetric monoidal functors
$I_n:\C_n\ra \C_{n+1}$ and $F_n:\C_n\ra Vec_K$.
We define the fundamental category of $W$ to be the direct limit:
$$\C_W:= \varinjlim_n \C_n$$
The functors $F_n$ then induce a faithful symmetric monoidal functor $F:\C_W\ra Vec_K$.

We would like to show that $\C_W$ is in fact abelian, and that it satisfies a certain universal property.
We begin with the following lemma:
\begin{lemma}
If $f:A\ra B$ is a morphism in $\C_{n-1}$, then the objects $I_{n}(K(f))$ and $K(I_{n-1}(f))$ 
are canonically isomorphic in $\C_{n+1}$.
\end{lemma}
\begin{proof}
in $\C_{n}$ we have a canonical map $i_f:K(f)\ra I_{n-1}(A)$ with $fi_f=0$.
By the construction of $\C_{n+1}$, as the kernel completion of $\C_n$, 
we have a canonical morphism $I_n(K(f))\ra K(I_{n-1}(f))$.
This morphism becomes invertible after applying $F_{n+1}$,
so by the construction of the kernel completion we know that it is invertible in $\C_{n+1}$ and therefore in $\C_W$.
\end{proof}
We will identify henceforth $I_n(K(f))$ with $K(I_{n-1}(f))$ via this canonical isomorphism,
and we will simply denote it by $K(f)$. 
More generally, since all the functors $I_n$ are faithful, we will consider them as identifications,
and we will consider all the objects and morphisms in $\C_n$ as objects and morphisms in $\C_W$. 

We next claim that $\C_W$ has all kernels and cokernels and that it is in fact abelian:
\begin{lemma} The object $K(f)$ is the kernel of $f$ in $\C_W$.
Similarly, $f$ has a cokernel in $\C_W$.
\end{lemma}
\begin{proof}
Assume that $g:C\ra A$ is a morphism in $\C_W$ such that $fg=0$.
We can consider $g$ and $f$ as morphisms in $\C_n$ for some $n$.
Then, since $\C_{n+1}$ is the kernel completion of $\C_n$,
we will have a unique morphism $g':C\ra K(f)$ in $\C_{n+1}$ 
(and therefore in $\C_W$) such that $i_fg'=g$.
For the cokernel statment, we use the fact that $\C_W$ is rigid, and following the lines of the proof of Theorem \ref{ker-comp},
we prove that the cokernel of $f$ is $K(f^*)^*$.
\end{proof}
Since we have kernels and cokernels in $\C_W$ we will write freely $Ker(f)$ and $Coker(f)$ from now on.
By the way we have defined the functors $F_n$ we get automatically that $F(Ker(f))=Ker(F(f))$.
By using rigidity we also have $F(Coker(f)) = Coker(F(f))$.
By the construction of $\C_W$ we know that if $f$ is a morphism in $\C_n$ such that $F_n(f)$ is invertible in $Vec_K$,
then $I_n(f)$ is invertible in $\C_{n+1}$. Therefore, every morphism $f$ in $\C_W$ for which $F(f)$ is invertible is already invertible in $\C_W$.
The next lemma proves that the category $\C_W$ is abelian:
\begin{lemma}\label{Fexact}
Let $\A$ be an additive category, let $\B$ be an abelian category and let $F:\A\ra \B$ be an additive faithful functor.
Assume that the category $\A$ has kernels and cokernels, that $F$ preserves kernels and cokernels, and that if
$F(f)$ is an isomorphism in $\B$ for some morphism $f$ in $\A$, then $f$ is already an isomorphism in $\A$.
Then $\A$ is an abelian category, and $F$ is exact.
\end{lemma}
\begin{proof}
In order to show that $\A$ is abelian, we need to prove that for a monomorphism $f:X\ra Y$ in $\A$ the induced morphism $X\ra Ker(Coker(f))$ is an isomorphism,
 and that if $g:X\ra Y$ is an epimorphism in $\A$ the induced morphism $Coker(Ker(g))\ra Y$ is an isomorphism.
 We will prove only the first statement, the proof of the second statement is similar.
Assume then that $f:X\ra Y$ is a monomorphism. This implies that $Ker(f)=0$ and therefore $Ker(F(f)) \cong F(Ker(f))=0$ and so $F(f)$ is also a monomorphism.
The natural map $\phi: X\ra Ker(Coker(f))$ becomes an isomorphism after applying $F$ and after applying the identification 
$F(Ker(Coker(f))) = Ker(F(Coker(f))) = Ker(Coker(F(f)))$. By the assumption on $F$ this implies that the morphism $\phi$ is an isomorphism.

In order to show that $F$ is exact, let $X\stackrel{f}{\ra} Y\stackrel{g}{\ra} Z$ be an exact sequence in $\A$.
Exactness means that $Im(f) = Ker(g)$. But since $Im(f) = Ker(Coker(f))$, we see that $F$ preserves also images. We thus have
$Im(F(f)) = F(Im(f)) = F(Ker(g)) = Ker(F(g))$ and so the sequence $F(X)\stackrel{F(f)}{\ra} F(Y)\stackrel{F(g)}{\ra} F(Z)$ 
is exact in $\B$ and $F$ is an exact functor.
\end{proof}
So $\C_W$ is an abelian symmetric monoidal rigid category and $F:\C_W\ra Vec_K$ is a faithful additive symmetric monoidal exact functor. 
The category $\C_W$ contains the structure $\wb{W}=(1,0)$ and we have $F(\wb{W}) = W$. For every tensor $x_i$ we have a morphism
$\wb{x_i}\in Hom_{\C_W}(\one, W^{p_i,q_i})$ such that $F(\wb{x_i}) = x_i$ (we identify $Hom_K(K,W^{p_i,q_i})$ with $W^{p_i,q_i}$).
We would like to show that $(\C_W,F,\wb{W})$ is universal with respect to these properties. 
We recall here Theorem \ref{universalproperty}.\\

\noindent\textbf{Theorem \ref{universalproperty}}\textit{
Let $K\subseteq T$ be an extension field, let $\A$ be an abelian symmetric monoidal rigid category, and let $G:\A\ra Vec_T$ 
be an exact additive faithful symmetric monoidal functor.
Assume that there is a structure $Z$ in $\A$ such that 
the structures $G(Z)$ and $W\ot_K T$ are isomorphic.
Then there exists a unique (up to isomorphism) exact faithful symmetric monoidal functor $\ti{F}:\C_W\ra \A$ 
such that $\ti{F}(\wb{W}) = Z$ (where equality here means equality of structures), 
and such that $G\ti{F}\cong i_{K,T}F$ where $i_{K,T}:Vec_K\ra Vec_T$ is the extension of scalars functor. }\\
\begin{proof}
Notice first that since $\A$ is rigid, tensor product with a given object in $\A$ is an exact functor.
Also, since $\A$ is abelian and $G:\A\ra Vec_T$ is exact and faithful, a morphism $f:A\ra B$ is invertible
if and only if $G(f)$ is invertible in $Vec_T$.

We shall construct the functor $\ti{F}$ step by step, starting from $\C_0$.
So for an object $(a,b)$ of $\C_0$ we define $\ti{F}((a,b)) = Z^{a,b}$ (since we would like $\ti{F}$ to be a monoidal functor,
this is the only possible definition).
We denote the tensors of the structure $W$ by $\{x_i\}$ and of the structure $Z$ by $\{y_i\}$.
The morphisms in $\C_0$ are constructed from the tensors $x_i$, 
the action of the symmetric group, pairing of $W$ and $W^*$ and concatenation of morphisms.
The same operations can also be done inside $\A$ using the object $Z$ and the morphisms $y_i$.
This implies that the image of $Hom_{\C_0}((a,b),(c,d))\ra Hom_K(W^{a,b},W^{c,d})$ is contained 
in the image of $Hom_{\A}(Z^{a,b},Z^{c,d})\ra Hom_K(G(Z^{a,b}),G(Z^{c,d})) = Hom_K(W^{a,b},W^{c,d})$.
We can apply now Lemma \ref{functorsdiagram} to the diagram
$$\xymatrix{ & \A\ar[d]^G \\ \C_0\ar[r]^{F_0} & Vec_T}$$
to define the action of $\ti{F}$ on morphisms in $\C_0$ (the isomorphism $\phi$ is given by $G(\ti{F}((a,b)) = G(Z^{a,b}) \cong W^{a,b} = F((a,b))$).
We also see that this functor is uniquely defined up to an isomorphism. 
We then extend uniquely the functor $\ti{F}$ to $\C_1$ by the universal property of additive envelopes.
The functors $\ti{F}$ extends uniquely to a functor $\C_2\ra \A$ and more generally to a functor $\C_n\ra \A$ for every $n$ 
by the universal property which appears in Theorem \ref{ker-comp}.
Since $\C_W$ is the direct limit of the categories $\C_n$ we get a unique faithful additive symmetric monoidal functor $\ti{F}:\C_W\ra \A$. 
Since this functor carries kernels to kernels, cokernels to cokernels, and satisfies that $\ti{F}(f)$ is an isomorphism if and only if $f$ is an isomorphism,
Lemma \ref{Fexact} shows that $\ti{F}$ is exact. 
\end{proof}
We finish this subsection with the following definition, which will be useful later:
\begin{definition}
We write $K_0=End_{\C_W}(\one)$. This is a subfield of $K$ which we call
the field of invariants of $(W,\{x_i\})$.
\end{definition}
\begin{remark}
There are different ways to construct an abelian category as an envelope of an additive category. 
See for example the paper \cite{Bel} and the Universal Property 2.10 in \cite{Krause}.
The construction we present here relies heavily on the presence of the functor $F$, 
and the fact that it gives us a way to interpret the objects and morphisms in $\C_1$ as objects and morphisms in the abelian category $Vec_K$,
and is therefore quite different from the constructions in \cite{Krause} and in \cite{Bel}, which do not use such a functor.
\end{remark}
\end{subsection}
\end{section}

\begin{section}{Field extensions and some basic properties of $\C_W$}\label{properties}
We thus have a $K_0$-linear rigid symmetric monoidal category $\C_W$ attached to the structure $W$.
This category is an invariant of the isomorphism type of $W$. 
We will next show that the category does not change when we take field extensions of the base field $K$,
and therefore structures which are forms of one another will have equivalent fundamental categories.
\begin{lemma}\label{extension}
Let $T$ be an extension field of $K$ and let $W$ be a structure defined over $K$.
We have a natural equivalence of symmetric monoidal $K_0$-linear categories $G:\C_{W\ot_K T}\ra \C_W$
between the fundamental category of $W$ and the fundamental category of the extension of scalars $W\ot_K T$.
Moreover, if $F_W:\C_W\rightarrow Vec_K$ and $F_{W\ot_K T}:\ti{\C}\rightarrow Vec_T$ are the two monoidal functors, and $i_{K,T}:Vec_K\rightarrow Vec_T$ 
is the extension of scalars functor (given by $i_{K,T}(V) = V\ot_K T$),
then we have a natural isomorphism of functors $i_{K,T}F_W\cong F_{W\ot_K T}G$
\end{lemma}
\begin{proof}
We consider the functor $F'=i_{K,T}F_W:\C_W\ra Vec_T$. The universal property of $\C_{W\ot_K T}$ gives us a
faithful symmetric monoidal additive functor $G:\C_{W\ot_K T}\ra \C_W$ such that $i_{K,T}F_W\cong F_{W\ot_K T}G$ and such that $G(\wb{W\ot_K T}) = \wb{W}$.
The universal property of $\C_W$ gives us a functor $H:\C_W\ra \C_{W\ot_K T}$ which satisfies $H(\wb{W}) = \wb{W\ot_K T}$. 
by using the universal property again we get that $GH\cong Id_{\C_W}$ and $HG\cong Id_{\C_{W\ot_K T}}$.
\end{proof}
\begin{corollary}\label{descent}
Let $W,K,T$ be as above. 
Assume that $L\subseteq T$ is another subfield and that $W\ot_K T$ has a form $W'$ over $L$.
Then the fundamental categories $\C_W$ and $\C_{W'}$ are equivalent, and $K_0\subseteq L$.
\end{corollary}
\begin{proof}
We assume that $W'\ot_{L} T\cong W\ot_K T$. 
Let $\ti{\C}$ be the fundamental category of $W\ot_K T$.
Then we have seen that both $\C_W$ and $\C_{W'}$ are equivalent to $\ti{\C}$.
If we write $K_0=End_{\C_W}(\one)$ and $K_1=End_{\C_{W'}}(\one)$ and $K_2 = End_{\ti{\C}}(\one)$ then
the equivalences between the categories shows that all three subfields $K_0,K_1,K_2$ 
are the same subfield of $T$, and in particular $K_0=K_1\subseteq L$.
\end{proof}	
The field of invariants $K_0$ is going to be a relatively small field in most cases.
More precisely, we claim the following lemma:
\begin{lemma}
If the set $\{x_i\}$ of structure-tensors $W$ is finite, 
then the transcendental degree of $K_0$ over $\Q$ is finite. In particular- $K_0$ is countable.
\end{lemma}
\begin{proof}
 We have seen that $K_0$ is contained in every field over which $W$ has a form.
Choose an arbitrary basis $\{w_j\}$ for $W$ over $K$. Denote the dual basis by $\{w^j\}$.
Each tensor $x_i\in W^{p_i,q_i}$ can be written as a linear combination of tensor products of $w_j$'s and $w^j$'s.
Let us denote by $K_1$ the subfield of $K$ which is generated by all the coefficients appearing in all these linear combinations.
Then $W$ has a form over $K_1$, and thus $K_0\subseteq K_1$.
Since the number of tensors is finite, $K_1$ has a finite transcendental degree over $\Q$.
The transcendental degree of $K_0$ over $\Q$ is thus finite as well.
\end{proof}

The next lemma will be used in Section \ref{identities} to prove that all polynomial identities of $W$ are already defined over $K_0$ 
(in case $W$ is an algebra or a comodule algebra).
\begin{lemma}\label{foridentities}
Let $X$, $Y$ be two objects of $\C_W$. Let $f_1,f_2,\ldots, f_n\in Hom_{\C_W}(X,Y)$ be non-zero morphisms. Assume that $F(f_1),\ldots, F(f_n)$ are linearly dependent over $K$.
Then $f_1,\ldots, f_n$ are linearly dependent over $K_0$.
\end{lemma}
\begin{proof}
Assume that $f_1,\ldots, f_n$ are morphisms in $Hom_{\C_W}(X,Y)$.
The functor $F$ induces a natural map $$\bigwedge^n F:\bigwedge^n Hom_{\C_W}(X,Y)\ra \bigwedge^n Hom_K(F(X),F(Y))$$
given by $f_1\wedge\ldots\wedge f_n\mapsto  F(f_1)\wedge\ldots \wedge F(f_n)$.
Since $F$ is faithful and symmetric, this map is well defined and injective. 
Therefore, if $f_1\ldots f_n$ are linearly independent over $K_0$ then $f_1\wedge\ldots\wedge f_n\neq 0$
and by the injectivity of $\bigwedge^n F$ we have that $F(f_1)\wedge\ldots \wedge F(f_n)\neq 0$ so that 
$F(f_1)\ldots F(f_n)$ are linearly independent over $K$.
\end{proof}
\end{section}

\begin{section}{The category $\C_W$ as a form of a representation category, and a proof of Theorem \ref{main1}}\label{cisform}
\label{categoryform}
From this section onwards we will assume that the field $K$ is algebraically closed. 
Since we can always extend scalars to the algebraic closure without altering the category $\C_W$, 
this will not be restrictive. On the other hand, it will be very useful when we will apply Deligne's Theory.

Let $G$ be the automorphism group of $(W,\{x_i\})$. That is $$G=\{g\in GL(W)|\forall i \, g(x_i)=x_i \}$$
(we have used here the fact that we have an induced action of $GL(W)$ on $W^{p,q}$).
Notice that $G$ is a closed subgroup of $GL(W)$ with respect to the Zariski Topology (so it is an affine algebraic group).
Our goal is to prove that $\C_W$ is a form of $Rep_K-G$. We will begin by recalling the following result of Deligne (see Theorem 1.12 and Theorem 7.1 in \cite{Deligne-tannaka}, 
and also Theorem 0.6 in \cite{Deligne-tensor} for a more general statement):
\begin{theorem}\label{Deligne}
Let $\D$ be a symmetric rigid $K$-linear monoidal category, tensor-generated by finitely many objects, where $K$ is an algebraically closed field of characteristic zero.
Assume that for every object $X\in \D$ we have $\bigwedge^n X=0$ for some $n>0$.
Then there exists a unique (up to isomorphism) fiber functor $F:\D\rightarrow Vec_K$.
Moreover, the group $G:=Aut_{\ot}(F)$ is an affine algebraic group over $K$, and we have a natural equivalence of $K$-linear symmetric monoidal categories 
$\D\rightarrow Rep_K-G$:
If $X$ is an object of $\D$ then the action of $G=Aut_{\ot}(F)$ on $F$ makes $F(X)$ a $G$-representation in a natural way.
\end{theorem}
\begin{remark} The fact that a fiber functor induces an equivalence between $\D$ and the representation category is known as Tannaka reconstruction, or Tannaka-Krein Duality.
\end{remark}
\begin{definition} 
We call $G$ the fundamental group of $\D$ (see also Section 8 of \cite{Deligne-tannaka}).
\end{definition}

In order to apply Deligne's Theorem, we need to extend scalars from out category $\C_W$ which is linear over $K_0$ to $K$. 
To do so, we begin with defining a new structure.
\begin{definition} 
Let $K_0\subseteq K'\subseteq K$ be an intermediate field.
For every $a\in K'$ we consider the tensor $y_a=a\in W^{0,0}=K$. 
The structure $W_{K'}$ is the structure given by the union of the original tensors in $W$ together with the tensors $y_a$. 
In other words, it is the structure $(W,\{x_i\}\cup \{y_a\})$.
\end{definition}
As for any other structure, we can construct the fundamental category $\C_{W_{K'}}$.
We write $F:\C_W\ra Vec_K$ and $F':\C_{W_{K'}}\ra Vec_K$ for the two fiber functors of the two fundamental categories. 
The fundamental category of $W'$ satisfies the following universal property:
\begin{lemma}\label{universalproperty2}
Let $\D$ be a $K'$-linear rigid symmetric monoidal category, and let $G:\D\ra Vec_K$ be a fiber functor.
Then we have an equivalence between the categories $\X:=Fun_{c,K_0}(\C_W,\D)$ and $\X':=Fun_{c,K'}(\C_{W_{K'}},\D)$
where $\X$ is the category of all $K_0$-linear symmetric monoidal exact faithful functors $H:\C_W\ra\D$ such that $GH\cong F$
and $\X'$ is the category of all $K'$-linear symmetric monoidal exact faithful functors $H':\C_{W_{K'}}\ra\D$ such that $GH'\cong F'$.
This equivalence is natural with respect to $K'$-libear symmetric monoidal exact faithful functors $J:\D\ra \D'$ of categories over $Vec_K$. 
\end{lemma}
\begin{proof}
A functor $H:\C_W\ra \D$ in the first category will give rise to a structure $\ti{W}:=H(\wb{W})$ of $\D$ which is of the same type as $W$,
and which lies above $W$ in $Vec_K$.
On the other hand, since $\D$ is a $K'$-linear category, the structure $\ti{W}$ also has the structure of $W_{K'}$.
By the universal property of Theorem \ref{universalproperty} for the category $\C_{W_{K'}}$, we get a functor $H':\C_{W_{K'}}\ra \D$ in the category $\X'$.
On the other hand, if $H':\C_{W_{K'}}\ra \D$ is a functor in $\X'$, then the structure $H(\wb{W_{K'}})$ is a structure in $\D$ of the same type as $W_{K'}$.
In particular, it has all the tensors $x_i$, so it is also a structure of the same type as $W$.
By using the universal property of Theorem \ref{universalproperty} again, we get a functor $H:\C_W\ra\D$ in $\X$.
Another application of the universal property will show us that these two maps are inverse to each other,
and a final application of the universal property will show us that this correspondence has the desired natural property.
\end{proof}
\begin{definition}
We call the category $\C_{W_{K'}}$ the extension of scalars of $\C_W$ to $K'$ and write 
$\C_W\ot_{K_0}K' := \C_{W_{K'}}$. 
\end{definition}
\begin{remark}
Deligne and Milne (see \cite{Deligne-Milne}) defined the extension of scalars in case $K'$ is a finite extension of $K_0$ in a different way: 
they considered $K'$ as an algebra in $\C_W$
(which is possible because $K'$ is an algebra in $Vec_{K_0}$ and $Vec_{K_0}$ is embedded in $\C_W$ by using the unit object), and they then considered
the category $K'-mod$ of $K'$ modules inside $\C_W$. This is a $K'$-linear rigid symmetric monoidal category. 
It is possible to show that the two definitions will give equivalent categories.
\end{remark}
Notice that by the universal property in Lemma \ref{universalproperty2}, 
if $K_0\subseteq K'\subseteq K''\subseteq K$, then there is a one-to-one correspondence between fiber functors $\C_W\rightarrow Vec_{K''}$ and fiber functors
$\C_W\ot_{K_0} K'\rightarrow Vec_{K''}$.
The following theorem is the first half of Theorem \ref{main1}
\begin{theorem}
The category $\C_W$ is a $K_0$-form of $Rep_K-G$. In other words, we have an equivalence of symmetric monoidal $K$-linear categories 
$\C_W\otimes_{K_0} K\rightarrow Rep_K-G$
\end{theorem}
\begin{proof}
The functor $F:\C_W\rightarrow Vec_K$ extends naturally to $F_K:\C_W\otimes_{K_0} K\rightarrow Vec_K$.
Let us write $\ti{G}=Aut(F_K)$. By Tannaka reconstruction we know that we have an equivalence between $\C_W\otimes_{K_0} K$ and $Rep_K-\ti{G}$. The equivalence is given
in the following way:
if $U$ is an object of $\C_W\otimes_{K_0} K$, then $F_K(U)\in Vec_K$ is a vector space, and if $g\in \ti{G}$, then we have $g_U:F_K(U)\rightarrow F_K(U)$.
This furnishes a structure of a $\ti{G}$-representation on $F_K(U)$.
We would like now to determine $\ti{G}$.
First of all, notice that the map $\ti{G}\rightarrow GL(F_K(\wb{W}))$ is one-to-one.
This is due to the following reason:
if $g\in\ti{G}$ acts trivially on $F_K(\wb{W})$ 
then by the fact that $F_K$ is an additive monoidal functor, $g$ acts trivially on all of the category $\C_1$ (see Section \ref{construction}).
It then acts trivially on all of $\C_W$, because all the other objects of $\C_W$ are derived from $\C_1$ as iterated kernels.
It thus follows that $g=1$. We can thus consider $\ti{G}$ as a subgroup of $GL(F_K(\wb{W})=GL(W)$.
But if $g\in \ti{G}$ then it follows that $g$ fixes all the tensors $x_i$ (because the way we have constructed the groups $X^{p,q}$).
Conversely, if $g$ fixes all the $x_i$'s, then it follows easily that all the vectors in $X^{p,q}$ for any $p$ and $q$ are $g$-invariant, and therefore $g$ induces
an action on the functor $F_1:\C_1\ra Vec_K$, and then, by induction, also on the functor $F:\C_W\ra Vec_K$.
So $\ti{G}=G$, and we have the desired result.
\end{proof}
\begin{remark}
 Due to the construction of $\C_W$ the field $K_0$ must be contained in every field over which $W$ has a form.
 In case $W$ has a form over $K_0$ itself, $K_0$ is usually referred to as a field of definition for $W$.
 However, there are cases in which $W$ will not have a form over $K_0$ (see Subsection \ref{nok0}).
 We do see that even though $W$ might not have a form over $K_0$, the category of representations of the automorphism group of $W$ will always have one.
\end{remark}

The above theorem shows us how we can reconstruct $W$ out of $\C_W$ and some additional data.
Indeed, we can think of $W$ as the object $\wb{W}$ of $\C_W$, and the tensors $x_i$ can be considered as morphisms $x_i\in Hom_{\C_W}(\one,\wb{W}^{\ot p}\ot (\wb{W}^*)^{\ot q})$.
The equivalence $\C_W\ot_{K_0}K\cong Rep_K-G$ gives us a fiber functor $F:\C_W\ot_{K_0}K\rightarrow Vec_K$, and this gives us the structure $(F(\wb{W}),\{F(x_i)\})$, 
which is isomorphic with $(W,\{x_i\})$.
We have used here the fact that $K$ is algebraically closed, and therefore there exists only one fiber functor on $\C_W\ot_{K_0}K$ (up to equivalence).
We record our result in the following lemma:
\begin{lemma}\label{reconstruction}
The structure $(W,\{x_i\})$ can be reconstructed from the following data: the category $\C_W$, the object $\wb{W}$ and the morphisms 
$x_i\in Hom_{\C_W}(\one,\wb{W}^{\ot p}\ot (\wb{W}^*)^{\ot q})$.
\end{lemma}

The next result gives us the connection between forms of $W$ and fiber functors. It finishes the proof of Theorem \ref{main1}.
\begin{theorem}
Let $K_0\subseteq K'\subseteq K$ be an intermediate field. There is a one-to-one correspondence between forms of $W$ over $K'$ and fiber functors 
$\C_W\rightarrow Vec_{K'}$.
\end{theorem}
\begin{proof}
Assume that $F':\C_W\rightarrow Vec_{K'}$ is a fiber functor.
Let $W':=F'(\wb{W})$. Then $W'$ is a vector space over $K'$, and if $x_i$ is a tensor of type $(p_i,q_i)$, then we have the tensors $y_i:=F'(x_i)\in W'^{p_i,q_i}$
(as mentioned earlier, we can think of the tensors $x_i$ as morphisms inside $\C_W$). We need to prove that $(W',\{y_i\})$ is indeed a form of $(W,x_i)$.
We thus need to prove that $(W'\ot_{K'} K,\{y_i\})\cong (W,\{x_i\})$.
The functors $F$ and $F'$ induce two fiber functors $\C_W\ot_{K_0} K\rightarrow Vec_K$.
Since $K$ is algebraically closed, we know from Theorem \ref{Deligne} that they are isomorphic.
But an isomorphism between them will induce an isomorphism $(W'\ot_{K'} K,\{y_i\})\cong (W,\{x_i\})$ as required.
In the other direction, assume that $(W',\{y_i\})$ is a form of $(W,\{x_i\})$.
Then we can construct the fundamental category $\D$ of $(W',\{y_i\})$.
But we have seen in Lemma \ref{extension} that this category depends only on the extension of scalars of $W'$ to $K$.
Therefore, we have an equivalence of $K_0$-linear rigid symmetric monoidal categories $\D\cong \C_W$.
Since $W'$ induces a fiber functor from $\D$ to $Vec_{K'}$, we get a functor from $\C_W$ to $Vec_{K'}$ as required.
\end{proof}
\end{section}

\begin{section}{Construction of the generic form, and a proof of Theorem \ref{main2}}\label{generic}
By the work of Deligne we know that if $K_0$ is an algebraically closed field,
then $\C_W$ is necessarily the representation category of some algebraic group.
However, $K_0$ is usually not algebraically closed. In this section we will use Deligne's theory, in order to deduce Theorem \ref{main2} and construct the generic form $\ti{W}$.

In order to prove the theorem, we will follow the original proof of Deligne. 
We will study algebras and modules inside the category $Ind(\C_W)$, and we will explain how we can use them in order to construct fiber functors.
Let then $A$ be a commutative algebra inside $Ind(\C_W)$.
As was mentioned in Section \ref{prel}, we can talk about $A$-modules inside the category $\C_W$.
We denote the category of all such modules (with $A$-module homomorphisms) by $A-mod$. This is again an abelian category.

We have a natural exact monoidal functor $F_A:\C_W\rightarrow A-mod$ given by $F_A(X)=A\ot X$, where the action of $A$ is on the first tensor factor.
We have $$Hom_{A-mod}(F_A(X),M)\cong Hom_{\C_W}(X,M).$$
Let $B:=Hom_{A-mod}(A,A)\cong Hom_{\C_W}(\one,A)$.
This is a $K_0$-algebra.
However, we can also view it as a subalgebra of $A$. Indeed, the algebra $B\ot\one$ is an algebra inside $\C_W$, and can be considered as a subobject 
(and in fact a subalgebra) of $A$.
Notice that if $F:\C_W\rightarrow Vec_{K'}$ is any fiber functor, then $F(B)$ will just be the extension of scalars $B\ot_{K_0}K'$.
Let us denote by $B-mod$ the category of $B$-modules (when $B$ is considered as an algebra in $Vec_{K_0}$, not in $\C_W$, and we also consider only modules in $Vec_{K_0}$).
We have the following lemma:
\begin{lemma}\label{trivialization}
Assume that $W\ot A$ is isomorphic in $\C_W$ with $A^n$ for some $n$.
Then the functor $F_B:\C_W\rightarrow B-mod$ given by $F_B(X) = Hom_{\C_W}(\one,F_A(X))$ is a monoidal functor.
\end{lemma}
\begin{proof}
We have a natural morphism $F_B(X)\ot F_B(Y)=Hom_{\C_W}(\one,X\ot A)\ot Hom_{\C_W}(\one,Y\ot A)\rightarrow Hom_{\C_W}(\one,X\ot Y\ot A)=F_B(X\ot Y)$. 
This morphism factors through $F_B(X)\ot_B F_B(Y)$. 
We will write the resulting morphism as $\beta_{X,Y}:F_B(X)\ot_B F_B(Y)\rightarrow F_B(X\ot Y)$.
Our goal is to prove that $\beta_{X,Y}$ is an isomorphism for every $X$ and $Y$.
We will prove it by induction on the subcategories $\C_i$.
Since $W\ot A\cong A^n$ we also have that $W^*\ot A\cong A^n$, and $W^{p,q}\ot A$ will also be a free $A$-module of finite rank for every $p$ and $q$.
Since the functor $$A-mod\rightarrow B-mod$$ $$X\mapsto Hom_{\C_W}(\one,X)$$ is a monoidal functor when restricted to the subcategory of $A$-modules of the form $\oplus A$, 
we have that $\beta_{X,Y}$ is an isomorphism for $X,Y\in \C_1$ (since the objects of $\C_1$ are direct sums of the objects of $\C_0$).

We now continue by induction.
Assume that $\beta_{X,Y}$ is an isomorphism for every $X,Y\in ob\C_i$. 
All the objects in $\C_{i+1}$ are formed as kernels of morphisms in $\C_i$.
We have already seen that every cokernel is a kernel (and vice versa).
So if $X\in Ob\C_{i+1}$ then there exists an exact sequence of the form $$0\rightarrow X\rightarrow Q\rightarrow W$$ where $Q,W\in Ob\C_i$. 
Assume that $Y$ is another object of $\C_i$. 
Then we also have the exact sequence $$0\rightarrow X\ot Y\rightarrow Q\ot Y\rightarrow W\ot Y.$$
The functor $-\ot A$ is exact, and the functor $Hom_{\C_W}(\one,-)$ is left exact.
By applying $F_B$, we thus get the following diagram, in which the rows are exact:\\
$\xymatrix{
0\ar[r] & F_B(X)\ot F_B(Y)\ar[d]^{\beta_{X,Y}}\ar[r] & F_B(Q)\ot F_B(Y)\ar[r]\ar[d]^{\beta_{Q,Y}} & F_B(W)\ot F_B(Y)\ar[d]^{\beta_{W,Y}} \\
0\ar[r] & F_B(X\ot Y)\ar[r] & F_B(Q\ot Y)\ar[r] & F_B(W\ot Y)} $\\
By the induction hypothesis we know that $\beta_{Q,Y}$ and $\beta_{W,Y}$ are isomorphisms. An easy diagram chase shows that $\beta_{X,Y}$ is also an isomorphism. 
In a similar way, we can now prove that $\beta_{X,Y}$ is an isomorphism for $X,Y\in ob\C_{i+1}$ and we are done.
\end{proof}
Notice, however, that the resulting functor $F:\C_W\rightarrow B-mod$ might fail to be exact.
We have the following lemma:
\begin{lemma}\label{split}
Assume that every short exact sequence $0\rightarrow X\rightarrow Y\rightarrow Z\rightarrow 0$ splits after tensoring with $A$.
Then the functor $F_B$ is exact.
\end{lemma}
\begin{proof}
This follows from the fact that the functor $F_A$ is exact, and that any additive functor between abelian categories is exact when restricted to split exact sequences.
\end{proof}

In the original work of Deligne, he constructed an algebra $A_D$ which satisfies the requirements of the two lemmas above.
This algebra $A_D$ will be a tensor product of localizations of symmetric algebras. It might be an infinite tensor product (see section 2.10-2.11 of \cite{Deligne-tensor}).
The algebra $A_D$ gives rise to a fiber functor $F_B:\C_W\rightarrow B-mod$.
If $\phi:B\rightarrow K_1$ is a homomorphism from $B$ to a field $K_1$ of characteristic zero, then
we can compose $F_B$ with the resulting functor $\phi:B-mod\rightarrow Vec_{K_1}$ to get a fiber functor from $\C_W$ to $Vec_{K_1}$
(since all the short exact sequence in $\C_W$ split in $B-mod$, the resulting functor is still exact).
\begin{definition}\label{classify}
We say that a commutative algebra $A$ inside $\C_W$ is a \textit{classifying algebra} if: \\
1. It satisfies the conditions of Lemmas \ref{trivialization} and \ref{split} \\
2. Every fiber functor from $\C_W$ to some extension field $K_1$ of $K_0$ arises from some homomorphism from $B$ to $K_1$
\end{definition}
\begin{proposition}\label{condclass}
Assume that $A$ satisfies the conditions of Lemmas \ref{trivialization} and \ref{split}, 
and that $A$ is a tensor product of localizations of symmetric algebras. Then $A$ also satisfies condition 2 of the definition above,
and therefore $A$ is a classifying algebra.
\end{proposition}
\begin{proof}
Let $F':\C_W\rightarrow Vec_{K_1}$ be a fiber functor.
Write $A$ as $$A=\bigotimes_i Sym(M_i)_{f_i}$$
where $M_i$ are objects of $\C_W$ and $f_i$ are nonzero elements in $Sym(M_i)$
Then we have that $F'(Sym(M_i)_{f_i}) = Sym(F'(M_i))_{F'(f_i)}$. The second algebra is a localization of a symmetric algebra over $K_1$.
We thus have a $K_1$-homomorphism $Sym(F'(M_i))_{F'(f_i)}\rightarrow K_1$. This is because $K_1$ is infinite, and therefore almost all homomorphisms
$Sym(F'(M_i))\rightarrow K_1$ will extend to any finite localization. 
By taking the tensor product, we get a homomorphism $\phi:F'(A)\rightarrow K_1$. This homomorphism restricts to $F'(B)=B\otimes K_1$ and therefore to $B$. 
The resulting homomorphism (which we denote by the same letter) $\phi:B\rightarrow K_1$ gives rise to a fiber functor $F'':\C_W\rightarrow Vec_{K_1}$. 
Moreover, the two functors are isomorphic.
Indeed, the equivalence is defined in the following way:
We have $F''(X)=Hom_{\C_W}(\one,X\ot A)\ot_B K_1$. We have a natural map 
$$Hom_{\C_W}(\one,X\ot A)\rightarrow Hom_{K_1}(K_1,F'(X)\ot F'(A))\cong F'(X)\ot F'(A)\stackrel{Id\ot \phi}{\rightarrow} F'(X).$$
This map factors through $Hom_{\C_W}(\one,X\ot A)\ot_B K_1$, and we get a natural transformation $F''\rightarrow F'$. 
By Deligne (see Section 2.7 in \cite{Deligne-tannaka}) we know that any natural tensor transformation between two fiber functors is an isomorphism, so we are done.
\end{proof}
\begin{remark}
In most cases which will be of interest for us the group $G$ will be reductive, and the category $\C_W$ will be semisimple.
The condition of Lemma \ref{split} will then be satisfied automatically.
\end{remark}
We would like to construct a concrete example of a classifying algebra in case $G$ is reductive. 
Assume that the dimension of $W$ is $n$. 
We take $n$ copies of $W$, $W_1,W_2,\ldots, W_n$ and $n$ copies of $W^*$, $W^*_1,\ldots, W^*_n$.
We write $Id_{i,j}\in W_i\ot W^*_j$ for the canonical element representing the identity map.
We write $A=Sym(W_1\oplus W_2\oplus\cdots\oplus W_n\oplus W^*_1\oplus\cdots\oplus W^*_n)/(Id_{i,j}-\delta_{i,j})_{i,j}$
We claim the following proposition:
\begin{proposition} 
If $G$ is reductive then the algebra $A$ is a classifying algebra for $\C_W$.
\end{proposition}
\begin{proof}
For every $i=1,\ldots, n$ we have a map $W\rightarrow A$ given by the inclusion of $W$ as $W_i$ in $A$, 
and we have a map $\one\rightarrow W\ot A$ given by the coevaluation of $W$, $\one\rightarrow W\ot W^*_i$.
By extension of scalars we get maps  $\rho_i:W\ot A\rightarrow A$ and $\nu_i:A\rightarrow W\ot A$.
The direct sum of these maps will give us maps $\rho:W\ot A\rightarrow A^n$ and $\nu:A^n\rightarrow W\ot A$.
Applying the original fiber functor $F$ reveals the fact that these two maps are mutually inverse to each other and that the algebra $A$ is non-zero.
(we use here the fact that $F$ extends naturally to $Ind(\C_W)$, and that it is faithful).
We thus see that $A$ satisfies the condition of Lemma \ref{trivialization}.
Since $G$ is reductive, $A$ satisfies the condition of Lemma \ref{split} trivially.
We notice that $A$ can be seen as a localization of a symmetric algebra.
Indeed, $A$ is equal to the localization of the subalgebra $Sym[(W_1)^* \oplus (W_2)^* \oplus\cdots\oplus (W_n)^* ]$
by the determinant polynomial. More explicitly: we have an identification of $(W_1)^*\ot (W_2)^*\ot\cdots\ot (W_n)^*$ with
$(W^*)^{\ot n}$. The action of $S_n$ on the last vector space gives us a one dimensional sub-object inside $(W_1)^*\ot (W_2)^*\ot\cdots\ot (W_n)^*$. 
This sub-object will correspond to the determinant polynomial.
It follows from \ref{condclass} that $A$ is a classifying algebra.
It follows that $B\ot_{K_0} K$ is the subalgebra of $G$-invariants of $A\ot K$.
The algebra $B\ot_{K_0} K$ is a finitely generated algebra, because $G$ is reductive. 
It then follows that $B$ itself is finitely generated over $K_0$.
\end{proof}
\begin{proof}[Proof of Theorem \ref{main2}]
Let now $A$ be a classifying algebra for $\C_W$. 
We write $B_W=Hom_{\C_W}(\one,A)$. 
We then have a fiber functor $F_B:\C_W\rightarrow B_W-mod$.
Consider the $B$-module $\ti{W}:=F_B(\wb{W})$.
Since $W\ot A\cong A^n$, we have automatically that $\ti{W}\cong B_W^n$ as $B_W$-modules.
For every $i$, $x_i$ can be considered as a morphism in $Hom_{\C_W}((q_i,0),(p_i,0))$. 
We can therefore consider $F(x_i):\ti{W}^{\ot q_i}\rightarrow \ti{W}^{\ot q_i}$.
This will give us the structure on $\ti{W}$.
Now, if $\phi:B_W\rightarrow K_1$ is a homomorphism of rings,
then we can consider the composition $\phi F_B:\C_W\rightarrow Vec_{K_1}$ which is a fiber functor.
We then have that $\phi F_B(\wb{W})=\ti{W}\ot_{B_W} K_1$ is a form of $W$. 
In the other direction, every form $W'$ of $W$ over $K_1$ will induce $F':\C_W\rightarrow Vec_{K_1}$.
By Proposition \ref{condclass} this functor arises from a homomorphism $B_W\rightarrow K_1$. 
In order to prove that $B_W$ has no zero divisors, it is enough to prove that $B_W\ot_{K_0} K$ has none. 
But $B_W\ot_{K_0} K = F(B_W)$ is a subalgebra of $F(A)$. 
Since the algebra $F(A)$ is the localization of a symmetric algebra, it is integral, and the same is true for $B_W$ as desired.

Finally, we need to prove that if the group $G$ is reductive, then $B_W$ is finitely generated.
Since $B_W=Hom_{\C_W}(\one,A)$ we have $B_W\ot_{K_0} K=F(B_W) = F(A)^G$.
Since the group $G$ is reductive and we can chose $A$ in such a way that the algebra $F(A)$ is finitely generated 
(it will be the tensor product of finitely many finite localization of symmetric algebras), we get that $F(A)^G$ is finitely generated over $K$.
From this we can easily deduce that $B_W$ itself is finitely generated.
Therefore, if $m$ is a maximal ideal of $B_W$, then $B_W/m$ is a finite extension field of $K_0$ (this follows from Hilbert's Nullstellensatz),
and we will get a form over a finite extension of $K_0$.
\end{proof}
\end{section}

\begin{section}{The action of the Galois group, and a proof of Theorem \ref{main4}}\label{sec:galois}
In this section we study the case in which $K/K_0$ is a Galois extension.
More generally, assume that $L\subseteq K_0$ is a finite extension, and that $K/L$ is Galois.
We write $\Ga=Gal(K/L)$. Let $(W,\{x_i\})$ be some structure defined over $K$ (as mentioned before, we assume now that $K$ is algebraically closed).

For $\gamma\in \Ga$, we write $^{\gamma} W$ for the following $K$-vector space: as an abelian group $^{\gamma}W=W$ and the twisted action of $K$ is given by 
$$x\cdot w = \gamma^{-1}(x)w.$$ 
The tensors $\{x_i\}$ will give us tensors $\{^{\gamma}x_i\}$ on $^{\gamma} W$. 
If $\{e_j\}$ is a basis for $W$, then we can write every tensor as a $K$-linear combination of tensor products of elements from the basis with elements from the dual basis.
The vector space $^{\gamma}W$ will then have the same basis, 
and the tensors $\{ ^{\gamma} x_i\}$ will be the tensors given by the action of $\gamma$ on the coefficients of the original tensors.

It is possible that the two structures $(W,\{x_i\})$ and $(^{\gamma}W,\{^{\gamma}x_i\})$ will not be isomorphic.
For example, if $G$ is a finite group, and $W=K^{\alpha}G$ is a $G$-graded algebra, 
then $^{\gamma}W$ will be the twisted group algebra $K^{\gamma(\alpha)}G$. These two graded algebras need not be isomorphic.
However, if $(W,\{x_i\})\cong (W',\{y_i\})$ then $(^{\gamma}W, \{^{^\gamma}x_i\})\cong (^{\gamma}W',\{^{\gamma}y_i\})$.
We thus have an action of $\Ga$ on all isomorphism classes of structures $(W,\{x_i\})$ 
where $W$ is a vector space of dimension $n$ and $\{x_i\}$ is a family of tensors of types $(p_i,q_i)$.

This action of the Galois group generalizes to categories. Indeed, if $\C$ is a $K_1$-linear category (where $K_0\subseteq K_1\subseteq K$) and $\gamma\in \Ga$,
then we can define a new $\gamma(K_1)$-linear category $^{\gamma}\C$ in the following way:
$^{\gamma}\C$ is exactly the same category as $\C$, with the exception of the action of ${\gamma}(K_1)$ on the Hom-sets:
if $X,Y$ are objects of $\C$, then we define $$Hom_{^{\gamma}\C}(X,Y) =\, ^{\gamma}Hom_{\C}(X,Y).$$
Notice that we have a natural equivalence of $\gamma(K_1)$-linear categories $$S_{\gamma}:^{\gamma}Vec_{K_1}\cong Vec_{\gamma(K_1)}$$ given by mapping $W$ to $^{\gamma}W$.

Let now $\gamma\in \Ga$. We consider the structure $^{\gamma}W$, which is also defined over $K$. Since $\C_W$ is a $K_0$-linear category,
we can twist scalars, and get the category $^{\gamma}\C_W$ which is $\gamma(K_0)$-linear.
We claim the following lemma:
\begin{lemma}
We have an equivalence of categories $E_{\gamma}:\, ^{\gamma}\C_W\cong C_{^{\gamma}W}$ which takes $W$ to $^{\gamma}W$.
Moreover, the two functors $$S_{\gamma} (^{\gamma}F_W):\,^{\gamma}\C_W\rightarrow\, ^{\gamma}Vec_K\rightarrow Vec_{K}\textrm { and }$$ 
$$F_{^{\gamma}W}E_{\gamma}:\, ^{\gamma}\C_W\rightarrow C_{^{\gamma}W}\rightarrow Vec_{K}$$ are equivalent.
\end{lemma}
\begin{proof}
The proof follows directly from the universal property of the categories $\C_{^{\gamma}W}$ and $\C_W$ from Theorem \ref{universalproperty}	
\end{proof}
\begin{proof}[Proof of Theorem \ref{main4}]
We will prove that $\gamma\in\Gamma$ fixes $K_0$ pointwise if and only if it fixes the isomorphism type of $W$
(by using Galois Correspondence, this proves the theorem).
Assume first that $\gamma$ fixes $K_0$ pointwise. 
This implies that the identity functor $Id:\,^{\gamma}\C_W\cong \C_W$, is an equivalence of $K_0$-linear categories.
By the previous lemma we have an equivalence between $\C_W$ and $C_{^{\gamma}W}$ which sends $W$ to $^{\gamma}W$
and $x_i$ to $^{\gamma}x_i$. This implies that $W\cong\, ^{\gamma}W$, as desired (since $K$ is algebraically closed).

On the other hand, assume that $\psi:W\cong \,^{\gamma} W$ is an isomorphism of the two structures.
The idea is that any $x\in K_0$ is an invariant of the isomorphism type of $W$. 
The corresponding invariant for $^{\gamma}W$ will be $\gamma(x)$.
But if $W\cong \,^{\gamma}W$ then it must hold that $x=\gamma(x)$.
More precisely, 
the isomorphism $\psi$ induces an equivalence of fundamental categories $\Psi:\C_{^{\gamma}W}\cong \C_W$.
When this equivalence is composed with the equivalence $E_{\gamma}$ from the previous lemma, we get an equivalence $H_{\gamma}:\, ^{\gamma}\C_W\ra \C_W$. 
We have functors $^{\gamma}F_W$ and $F_W$ from $^{\gamma}\C_W$ and $\C_W$ into $Vec_K$,
and we have an isomorphism of functors $\mu:H_{\gamma}F_W\cong \,^{\gamma}F_W$.
This implies that the two homomorphisms of rings induced by the functors $F_W$, $^{\gamma}F_W$ and 
$H$, $\phi_1:\gamma(K_0)=End_{^{\gamma}\C_W}(\one)\ra End_{\C_W}(\one)\ra End_K(\one)=K$
and $\phi_2:End_{^{\gamma}\C_W}(\one)\ra End_K(\one)$ are equal up to conjugation by an element in $End_K(\one)=K$.
But since $K$ is commutative, this means that $\phi_1=\phi_2$.
The homomorphism $\phi_1$ is given by the natural inclusion followed by $\gamma^{-1}$.
The homomorphism $\phi_2$ is just the natural inclusion.
This implies that $K_0$ is fixed under $\gamma$, and we are done.
\end{proof}
The classical descent theory gives us a description of all the forms of $W$ over $K_1$, where $L\subseteq K_1\subseteq K$ in the following way:
let $H_1<\Ga$ be the stabilizer of $K_1$. 
If $W$ has a form over $K_1$
then $W$ has a basis with respect to which all the structure constants are in $K_1$, and so 
$^{\gamma}W\cong W$ for every $\gamma\in H_1$. 

An additive map $\phi:W\rightarrow W$ is called \textit{$\gamma$-linear} if it satisfies $$\forall x\in K \qquad \phi(\gamma(x)\cdot w) = x\phi(w).$$ 
We consider the group $\ti{H}$ of all invertible additive maps $\phi:W\rightarrow W$ which are $\gamma$-linear for some $\gamma\in H_1$,
and which satisfy $\forall i \;\phi(x_i)=x_i$.
Assuming that $H_1$ stabilizes the isomorphism type of $W$, we have a short exact sequence
\begin{equation}\label{ses-form}1\rightarrow G\rightarrow \ti{H}\rightarrow H_1\rightarrow 1\end{equation}
where $G$ is the automorphism group of the structure $W$. 
This sequence splits if and only if $W$ has a form over $K_1$.
Moreover, the different forms of $W$ correspond to different splittings 
(where two splitting are considered to be equivalent if they differ by conjugation by an element of $G$).

We thus see that we have two conditions for $W$ to have a form over $K_1$: firstly the group $H_1$ should stabilize the isomorphism type of $W$,
and secondly the above short exact sequence should split. 
The discussion we have here, together with Theorem \ref{main1}, shows us that if the first condition holds then 
we already have a form of the category $Rep_K-G$ over $K_1$. 
However, the form we get over $K_1$ might not have a fiber functor.
The obstruction to the existence of a form of the fiber functor is exactly the splitting of the sequence in (\ref{ses-form}), by Theorem \ref{main1}.

\end{section}

\begin{section}{relation to polynomial identities}\label{identities}
In this section we assume that our object $W$ is an associative algebra or an $H$-comodule algebra,
where $H$ is some finite dimensional Hopf algerba. 
We mention that the results of this section can also be applied to identities of non-associative algebras (e.g. for Lie algebras or Jordan algebras).
The formulation will just be more complicated.
A polynomial identity of an associative algebra $W$ is a noncommutative polynomial $f(X_1,X_2\ldots,X_n)$ such that $f(v_1,v_2\ldots,v_n)=0$ for any $v_1,v_2\ldots,v_n\in W$.
For example, if $W$ is a commutative algebra, then $f(X_1,X_2)=X_1X_2-X_2X_1$ is a polynomial identity of $W$. 
Another example for polynomial identity is the famous Amitsur-Levitsky identity:
if $W=M_n(K)$, then the polynomial $$f(X_1,\ldots, X_{2n})=\sum_{\sigma\in S_{2n}}(-1)^{\sigma}X_{\sigma(1)}X_{\sigma(2)}\cdots X_{\sigma(2n)}$$ is an identity, 
and $W$ does not have polynomial identities of lower degree then $2n$. 
Notice that both the Amitsur-Levitsky identity and the commutation identity are multilinear polynomials (that is- in all monomials every variable appears exactly once).
In characteristic zero it is possible to prove that all polynomial identities are derived from multilinear identities, and therefore we will focus on them.

We can think of a polynomial identity in the following way:
Let us denote by $m:W\ot W\rightarrow W$ the multiplication on $W$.
We write $m^{n-1}:W^{\ot n}\rightarrow W$ for the iterated multiplication.
The symmetric group $S_n$ acts on $Hom_{\C_W}(W^{\ot n},W)$.
We can therefore look on the sub $S_n$-module  of $Hom_{\C_W}(W^{\ot n},W)$ generated by $m^{n-1}$.
A polynomial identity of degree $n$ will then be the same as a relation of the form
$$\sum_{\sigma\in S_n} a_{\sigma}\sigma\cdot m^{n-1}$$ where $a_{\sigma}\in K$.
Indeed, such a relation corresponds to the polynomial identity $\sum_{\sigma\in S_n}a_{\sigma}X_{\sigma(1)}\cdots X_{\sigma(n)}=0$.

Let now $H$ be a finite dimensional Hopf algebra defined over a subfield $k\subseteq K$.
An $H$-comodule algebra can be thought of as an algebra $W$ together with an action of $H^*$, such that $$\forall f\in H^* \, f(a\cdot b) = f_1(a)\cdot f_2(b),$$
where we use here the Sweedler Notation $$\Delta(f) = f_1\ot f_2.$$
We recall here the definition of $H$-identities from \cite{Kassel2} (this definition is slightly different from the one in \cite{Alkass}).
For every $i$, let $X^H_i$ be a copy of the vector space $H$.
We will denote the element in $X^H_i$ which corresponds to $h$ by $X^h_i$.
The tensor algebra $\T=T(\oplus_i X^H_i)$ is an $H$-comodule algebra, where the coaction is given on the generators by:
$$\rho(X^h_i) = X^{h_1}_i\ot h_2.$$
an element $P\in \T$ is a \textit{graded identity} of $W$ if for every homomorphism $\phi:\T\rightarrow W$ of $H$-comodule algebras it holds that $\phi(P)=0$.
We would like to write the identities as linear relations on morphisms in our category.
Since $H$ is finite dimensional, it is known that $H$ is isomorphic with $H^*$ when considered as a left $H^*$-module (or a right $H$-comodule). 
A canonical choice of a basis element will be the left integral $\l$ of $H$, which is unique up to a nonzero scalar. 
Any homomorphism of $H$-algebras $\T\rightarrow W$ is uniquely defined by its restriction to $\oplus_i X^H_i$. Its restriction to $X^H_i$ will be a map of $H$-comodules,
and therefore it will be uniquely defined by the image of $X^{\l}_i$.
Let us write $\{f^j\}$ for a basis of $H^*$. An identity will thus be a noncommutative polynomial $P$
in the variables $f^j\cdot X^{\l}_i$, which will vanish upon any instance of $X^{\l}_i\mapsto v_i\in W$.
A multilinearization shows that this identity is equivalent to a noncommutative polynomial in the variables $\{f^j\cdot X^{\l}_i\}$, 
in which every monomial contains $X^{\l}_i$ exactly once.

The algebra $((H^*)^{op})^{\ot n}$ acts on $Hom_{\C}(W^{\ot n}, W)$ by its action on the tensor factors of $W^{\ot n}$. 
The group $S_n$ acts on the same space as before- by permuting the tensor factors of $W^{\ot n}$. 
Together, we get an action of the crossed product algebra $H_n:=((H^*)^{op})^{\ot n} * S_n$ 
where the action of $S_n$ is by permuting the tensor factors of $((H^*)^{op})^{\ot n}$.
We conclude this discussion in the following lemma:
\begin{lemma}
Let $\{t_i\}$ be a basis for $H_n$ over $k$. An $H$-polynomial identity of $W$ of degree $n$ is equivalent to a linear relation of the form
$$\sum_i a_i t_i\cdot m^{n-1}=0.$$\end{lemma}

The conclusion of this is that both regular polynomial identities and $H$-polynomial identities can be understood as linear relations between morphisms in the category $\C_W$.
We are now ready to prove Theorem \ref{main5}.
\begin{proof}[Proof of Theorem \ref{main5}]
Assume that $W$ is an $H$-comodule algebra over a field $K$, and that the Hopf algebra $H$ is defined already over a subfield $k$ of $K$.
We have seen that $H$-polynomial identities correspond to the vanishing of linear combinations of morphisms in $\C_W$ over $K$. 
We have proved in Lemma \ref{foridentities} that if morphisms in $\C_W$ are linearly dependent over $K$,
then they are linearly dependent already over $K_0$. This finishes the proof.
\end{proof}
Notice that polynomial identities give us in general less information on the algebra than the category $\C_W$. 
Indeed, the polynomial identities of $W$ and of $W\oplus W$ are the same, and therefore the polynomial identities cannot define the isomorphism type of the algebra.
The polynomial identities do define the algebra if one makes some extra assumptions on the algebra.
In \cite{AH} Aljadeff and Haile proved that if $W$ is a simple $H$-comodule algebra where $H=kG$ is a group algebra, then the identities of $W$ determine $W$.
In \cite{Kassel2} Kassel proved that $H$-identities can be used to distinguish between isomorphism classes of different Hopf Galois extensions of the ground field 
for the Taft algebras $H_{n^2}$ and for the Hopf algebras $E(n)$.

We give here an example:
Let $n$ be a natural number, and let $G=C_n\times C_n$ be generated by $g$ and $h$.
Let $\alpha$ be the two-cocycle on $G$ defined by $\alpha(g^ih^j,g^kh^l)=\zeta^{jk}$ where
$\zeta$ is a primitive $n$-th root of unity.
Then the polynomial identity $X_hX_g=\zeta X_gX_h$ is defined over $\Q(\zeta)$, and we will prove in Section \ref{examples2}
that $\Q(\zeta)=K_0$.
In Section \ref{examples0} we will see an example for an associative algebra in which all the polynomial identities are already defined over a proper subfield of $K_0$.
\end{section}
\begin{section}{First examples}\label{examples4}
We begin with the example $(W,\{x_i\})$ in which the set of tensors $\{x_i\}$ is empty. 
In this case, the group $G$ is the entire group $GL(W)$, and the objects in the category $\C_1$ are all direct sums of objects of the form $W^{p,q}$.
The only nontrivial endomorphisms of $W^{p,0}$ will be those arising from the action of the symmetric group on $W^{p,0}$ by permuting the tensor factors.
One can show that in this case the Karoubian envelope (a.k.a. idempotent completion) of $\C_1$ will already be an abelian category,
where the Karoubian envelope is the category we get of $\C_1$ by adding to it only kernels and cokernels of idempotent morphisms
(for this, we use the fact that the finite group $S_p$ acts on $W^{p,q}$ in a semisimple way).
This implies that the Karoubian envelope of $\C_1$ is equivalent to $\C_W$ in this case.
We then recover the well known Schur-Weyl duality, which says that the endomorphism ring of $W^{p,0}$ as a $GL(W)$-module 
is generated by the action of the symmetric group.
We also recover the fact that all the finite dimensional rational representations of $GL(W)$ can be constructed from the canonical representation $W$ by taking tensor products,
subrepresentations, quotients, and duals.

We continue with the case where the set $\{x_i\}$ contains exactly one element of type $(1,1)$.
Let $W$ then be a finite dimensional vector space over $K$, and let $T:W\ra W$ be an endomorphism.
We consider the structure $(W,\{T\})$. 
If we identify $End_K(W)$ with $W\ot W^*$, then the linear functional 
$$End_K(W)\cong W\ot W^*\stackrel{c_{W,W^*}}{\ra} W^*\ot W \stackrel{ev}{\ra} K$$ is the trace function (it is easy to prove this, for example by choosing a basis for $W$).
We thus see that the field $K_0$ will contain all the elements $tr(T), tr(T^2),\ldots$.
The coefficients of the characteristic polynomial can be written as polynomials over $\Q$ in $tr(T^i)$.
We thus see that the field of invariants $K_0$ contains all the coefficients $c_i$ of the characteristic polynomial.
In case $T$ is semisimple, it is easy to show that $K_0= \Q(c_i)$, $(W,\{T\})$ has a form over $K_0$ and it is in fact unique.
In case $T$ is not semisimple, the field $K_0$ might be bigger. 
However, by studying the possible Jordan decompositions one can prove that it is still true that $(W,\{T\})$ has a unique form over $K_0$.
For example, if $K$ contains $\sqrt{2}$, $W=K^4$ and $$ T = \begin{pmatrix}
\sqrt{2} & 0 & 0 & 0 \\ 
0 & \sqrt{2} & 0 & 0 \\
0 & 0 & -\sqrt{2} & 1 \\
0 & 0 & 0 & -\sqrt{2}
\end{pmatrix}
$$
then the characteristic polynomial of $T$ is $(x^2-2)^2$, and all the coefficients are in $\Q$.
However, the image of $T^2-2$ is the one dimensional space spanned by $e_3$. 
The action of $T$ on that space will be by multiplication by $-\sqrt{2}$, and therefore $\sqrt{2}\in K_0$ 
(and since $W$ has a form over $\Q(\sqrt{2})$ we get that $K_0=\Q(\sqrt{2})$).
\end{section}

\begin{section}{Example: the field of definition of an associative algebra}\label{examples0}
Let $W$ be the following three dimensional algebra defined over $\Q(a)$
(where $a\neq 1,0$ can be algebraic or transcendental over $\Q$):
$$W=span\{x,y,z\}$$
$$xz = zx = yz = zy = z^2 = 0$$
$$z= x^2 = y^2 = xy = a^{-1}yx$$
Notice that $W$ does not have a unit, and that it is nilpotent of rank 3.
We will construct certain objects and morphisms in the category $\C_W$,
and we will prove that $a\in K_0$. This will prove that $K_0=\Q(a)$, 
since the algebra $W$ has a form over that field.

We denote the multiplication map by $m:W\ot W\rightarrow W$.
Consider first the subobject $W^2=Im(m) = span(z)$.
The multiplication induces a map $W/W^2\ot W/W^2\rightarrow W^2$.
This gives us two morphisms $W/W^2\rightarrow W^2\ot (W/W^2)^*$.
If we denote the dual basis $\{\bar{x},\bar{y}\}$ of $W/W^2$ by $\{e,f\}$, we get that these maps are invertible and given by:
$$T_1(\bar{x})= e\ot z + f\ot z$$
$$T_1(\bar{y})= ae\ot z + f\ot z$$
and 
$$T_2(\bar{x}) = e\ot z + af\ot z$$
$$T_2(\bar{y})= e\ot z + f\ot z.$$
The composition $T_2^{-1}T_1$ will give us a morphism $W/W^2\rightarrow W/W^2$.
A direct calculation shows that the trace of this morphism is $a+1$. This implies that $a\in K_0$ as desired.

The element $a$ cannot be seen via the polynomial identities of $W$.
Indeed, since $W$ is a nilpotent algebra of rank 3, any monomial of rank $\geq 3$ 
will be an identity, and in degree 2 there is no $b$ such that the polynomial $f(X_1,X_2)=X_1X_2+bX_2X_1$ is an identity.
All the polynomial identities are therefore already defined over $\Q$.
\end{section}

\begin{section}{Example: central simple algebras}\label{examples1}
\begin{subsection}{A splitting field for a central simple algebra}\label{nok0}
Let $D$ be a central simple algebra of dimension $n^2$ over a field $k$ of characteristic zero.
The algebra $D$ splits over an algebraic extension $L$ of $k$ if and only if
$D\ot_k L$ has a representation of dimension $n$ over $L$.
Assume then that $L$ is such an extension, and that $V$ is such a representation.
For every $d\in D$, we have a tensor $x_d\in V\ot V^*$ which gives the action of $d$ on $V$.
We construct the fundamental category for $(V,\{x_d\})$.
The field $K_0$ must include $k$ (since the traces of the tensors $x_d$ are in $K_0$),
and will in fact coincide with it.
The group $G$ will be $\mathbb{G}_m$, the multiplicative group,
since the only elements in $GL(V)$ commuting with all the $x_d$ tensors will be scalar multiplications.
The resulting short exact sequence we will get (see Section \ref{sec:galois}, here $\Gamma= Gal(L/K)$)
$$1\rightarrow \mathbb{G}_m\rightarrow \ti{G}\rightarrow\Ga\rightarrow 1$$
will correspond to an element in $H^2(\Ga,\mathbb{G}_m)$ which is the class of $[D]$ in $Br(L/k)$.
As a classifying algebra we can take $$A=k[V\oplus V^*]/(f(v)=1).$$
The invariant subalgebra will then be $$B=Hom_{\C}(\one,A)= k[D]/(tr(d)=1,rank(d)=1).$$
A point $B\rightarrow K_1$ will give us an element $e$ in $D_{K_1}$ which is an idempotent and for which $dim_{K_1}(D_{K_1}e)=n$.
The representation $D_{K_1}e$ will then be the desired representation.
In particular, if $D$ does not split over $k=K_0$, then we will not have a fiber functor $\C_W\ra Vec_k$.
\end{subsection}

\begin{subsection}{Central simple algebras and generic division algebras}
Let $n$ be a natural number, and let $W=M_n(\Q)$ be the $n\times n$ matrix algebra over $\Q$.
Let us denote by $m$ the multiplication map $m:W\ot W\rightarrow W$.
The fundamental category of $(W,\{m\})$ will give us a generic form of the matrix algebra.
This will in fact give us a localization of the generic division algebra which appears in the work of Procesi (see \cite{Formanek}).
The group $G$ here will be $PGL_n$, which is reductive. We will therefore receive a finitely generated commutative $\Q$-algebra $B$,
and a $B$-algebra $\ti{W}$ which is free of rank $n^2$ as a $B$-module, with the following property:
for every homomorphism $\phi:B\rightarrow K$ from $B$ into a field $K$ of characteristic zero the algebra $W_{\phi}=W\ot_B K$
will be a central simple $K$-algebra of dimension $n^2$, and every central simple algebra will be received in this way.

Before constructing explicitly the generic division algebra, let us look on some morphisms in $\C_W$.
The multiplication $m$ is an element in $Hom_{\C_W}(W\ot W,W)\subseteq W\ot W^*\ot W^*$. By pairing $W$ with one of the copies of $W^*$
we get an element in $Hom_{\C_W}(W,\one)\subseteq W^*$.
A direct calculation shows that this element will be the trace of the (left or right) regular representation of $W$
(this will be true for any algebra). By multiplying by $\frac{1}{n}$ we get the usual trace.
Now, if $t$ is a natural number, and $\sigma\in S_t$ is given by $\sigma=(i_1,i_2,\ldots i_r)(j_1,j_2,\ldots,j_s)...$
then we have the morphism $T_{\sigma}\in Hom_{\C_W}(W^{\ot t},\one)$ given by $$T_{\sigma}(X_1\ot X_2\ldots,X_t) =
tr(X_{i_1}X_{i_2}\ldots X_{i_r})tr(X_{j_1}X_{j_2}\ldots X_{j_s})... .$$
By the work \cite{Procesi} of Procesi we know that these are all the $PGL_n$-invariants.

We would like to construct a classifying algebra for $W$.
We will take the classifying algebra to be a localization of $Sym[W^*\oplus W^*]$.
The algebra $Sym[W^*\oplus W^*]$ will not be a classifying algebra itself, because it has too many points.
However, it will be a classifying algebra after we localize by a specific polynomial.
We begin with recalling Lemma 14 from \cite{Formanek2}
\begin{lemma}
Let $M_1,M_2,\ldots M_{n^2}$ be $n^2$ $n\times n$ matrices over a field $K$ of characteristic zero.
Then they will form a basis for $M_n(K)$ if and only if
$$
f(M_1,\ldots,M_{n^2})=$$ $$\sum_{\sigma\in S_{n^2}}(-1)^{\sigma}tr(M_{\sigma(1)})tr(M_{\sigma(2)}M_{\sigma(3)}M_{\sigma(4)})\cdots
tr(M_{\sigma(n^2-2n+2)}\cdots M_{\sigma(n^2)})$$ is nonzero.
\end{lemma}
This enables us to construct a classifying algebra. 
Indeed, We can take $Sym[W^*\oplus W^*\oplus \cdots\oplus W^*]_f$ (the coordinate algebra on the space of $n^2$ matrices, 
localized at the polynomial which says that they form a basis).
We can also get a smaller classifying algebra: we consider $Sym[W^*\oplus W^*]$, the coordinate algebra for the space of two $n\times n$ 
matrices which we shall denote $X$ and $Y$.
We write $M_{ni+j}=X^iY^j$ for $0\leq i,j \leq n-1$ and we define $D(X,Y):=f(M_1,\ldots, M_{n^2})$.
The localization $A:=Sym[W^*\oplus W^*]_D$ will thus give us a classifying algebra.
The reason for this is the following: The algebra $A$ will have $W^*\oplus W^*$ in degree one.
These two copies of $W^*$ will give us two morphisms $\phi_X,\phi_Y:\one\rightarrow W\ot A$ by the coevaluation of $W$.
By using the multiplication in $W$, we get $n^2$ maps $\phi_{X^iY^j}:\one\rightarrow W\ot A$.
We can extend scalars, take the direct sum and get a map $A^{n^2}\rightarrow W\ot A$.
The fact that $D$ is invertible implies that this map is invertible, and therefore $A$ is a classifying algebra for $\C_W$.
The resulting algebra $\ti{W}=Hom_{\C_W}(\one,W\ot A)$ is an Azumaya algebra which is free of rank $n^2$ over its center, $Hom_{\C_W}(\one,A)=\Q[W\oplus W]^{PGL_n}_D$.
This algebra will specialize to any central simple algebra of dimension $n^2$ over any field of characteristic zero.
This algebra is a localization of the algebra $\bar{R}$ introduced by Procesi. See the paper \cite{Formanek} by Formanek for more details.
\end{subsection}
\end{section}

\begin{section}{Examples: comodule algebras}\label{examples2}
Let $H$ be a finite dimensional $k$-Hopf algebra where $k\subseteq K$ is a subfield.
An $H_K:=H\ot_k K$-comodule algebra over $K$ is a $K$-algebra $W$ equipped with a right coaction of $H_K$: $\rho:W\rightarrow W\ot_K H_K$ 
which is also an algebra map.
For example, if $H=kG$ is a group algebra, then an $H$-comodule algebra is a $G$-graded algebra.
If $H=(kG)^*$, then an $H$-comodule algebra is an algebra $W$ equipped with an action of $G$ by algebra automorphisms.
Comodule algebras play an important role in the theory of Hopf algebras.
Of particular importance are comodule algebras of the form $^{\alpha}H$, where $\alpha:H\ot H\rightarrow K$ is some (convolution invertible) two-cocycle on $H$.
These algebras are identical with $H$ as $H$-comodules, and their multiplication is given by the formula
$$x\cdot_{\alpha} y = \alpha(x_1,y_1)x_2y_2.$$
The two-cocycle condition on $\alpha$ is equivalent to the associativity of the algebra.
It is known that any comodule algebra which is isomorphic with $H$ as an $H$-comodule is of this form.

We will now use the fundamental category to study comodule algebras.
We will consider the following three cases:
Group algebras, Taft algebras, and product of Taft algebras.
All our constructions can be generalized easily to the Hopf algebras $E(n)$ and to the monomial Hopf algebras.
We will explain, for $H$ a Taft algebra or a product of Taft algebras, how can one classify all the cocycles on $H$ by using the fundamental category
(this classification is known. See for example \cite{Masuoka} for the classification of two-cocycles on Sweedler's Hopf Algebra, which is the Taft Hopf algebra in dimension 4).
In \cite{Alkass} Aljadeff and Kassel have constructed an algebra $\A^{\alpha}_H$ which is a generic form of $^{\alpha}H$
(they have proved that it specialize to any form of $^{\alpha}H$ and the results of \cite{Kasmas} show that it specialize only to forms of $^{\alpha}H$).
In \cite{Iyerkas} Iyer and Kassel have studied the algebra $\B^{\alpha}_H=(\A^{\alpha}_H)^{coH}$ for Taft algebras, monomial Hopf algebras and the $E(n)$ algebras.
The construction we present here will give us a generic form over a basis of smaller Krull dimension (for the case of Taft algebras).

So let $\alpha$ be a two-cocycle on $H$ with values in $K$.
The fundamental category $\C_W$ of $W=\,^{\alpha}H$ will thus be constructed from the following tensors:\\
1. The multiplication $m:W\ot W\rightarrow W$.\\
2. For every $f\in H^*$ the action $T_f:W\rightarrow W$.\\ 
In terms of the map $\rho:W\ra W\ot H$, the map $T_f$ is given by $T_f=(Id\ot f)\rho$.
Let us determine first the fundamental group of the category. 
Since $W\cong H\cong H^*$ as an $H^*$-module, the only maps $W\rightarrow W$ which commute with the tensors $T_f$ will be of the form $h\mapsto g(h_1)h_2$
(where we identify $W$ with $H$). Now, such an element $g\in H^*$ will commute with the multiplication if and only if 
$g$ is a group like element in the dual of the twisted Hopf algebra
$^{\alpha}H^{\alpha^{-1}}$. So $\C_W$ is a $K_0$-form of $Rep_K-G((^{\alpha}H^{\alpha^{-1}})^*)$.
Notice in particular that this group is finite. 
Because we consider all tensors arising from the action of $H^*$, we will have in particular the action by all the scalars in $k$.
This shows that $k$ is necessarily contained in $K_0$, the field of invariants of $\C_W$. The field $K_0$ might be bigger though.

\begin{subsection}{Group algebras}
Let $G$ be a finite group, and let $H=\Q G$.
A two-cocycle on $H$ with values in $K$ will be the familiar object from group cohomology, namely a function $\alpha:G\times G\rightarrow K^{\mult}$ such that
$\alpha(x,y)\alpha(xy,z)=\alpha(y,z)\alpha(x,yz)$ for every $x,y,z\in G$.
Two two-cocycles $\alpha$ and $\beta$ are equivalent (or cohomologous) in case there is a function $\lambda:G\rightarrow K^{\mult}$ such that
$\alpha(x,y)=\beta(x,y)\lambda(x)\lambda(y)\lambda^{-1}(xy)$ for every $x,y\in G$.
The twisted group algebra $W=K^{\alpha}G$ has a basis $\{U_g\}_{g\in G}$ and the multiplication is defined by the formula 
$$U_x\cdot U_y = \alpha(x,y)U_{xy} \textrm{ for }x,y\in G.$$
The coaction of $KG$ is given by $\rho(U_g)=U_g\ot g$.
In other words, the action of $e_g\in (KG)^*$ is the projection $e_g:U_x\mapsto \delta_{x,g}U_x$.

We begin by describing some objects and morphisms in the category $\C_W$.
We will denote the image of $e_g$ in $W$ by $W_g$. We have an isomorphism $\one\rightarrow W_1$ given by sending $1\in K$ to the identity of $W$.
We will identify $\one$ and $W_1$ henceforth.
For every $g\in G$, the restriction of the multiplication map $W_g\ot W_{g^{-1}}\rightarrow \one$ will give us an isomorphism between
$W_g^*$ and $W_{g^{-1}}$.
The coevaluation will then be 
$$coev_g:\one\rightarrow W_g\ot W_{g^{-1}}$$ $$1\mapsto U_g\ot U_g^{-1}.$$
So for every $g,h\in G$, we have the following map in $\C_W$:
$$d_{g,h}:\one\rightarrow W_g\ot W_{g^{-1}}\ot W_h\ot W_{h^{-1}} \rightarrow W_g\ot W_h \ot W_{g^{-1}}\ot W_{h^{-1}}\rightarrow W_{ghg^{-1}h^{-1}}$$
which sends 1 to $w_{g,h}=U_gU_hU_g^{-1}U_h^{-1}=c_{g,h}U_{ghg^{-1}h^{-1}}$ for some $c_{g,h}\in K^{\mult}$.
This gives us an isomorphism in $\C_W$ between $\one$ and $W_{ghg^{-1}h^{-1}}$. 
This means that $w_{g,h}$ must be contained in any form of $W$.
This also means that if $$g_1h_1g_1^{-1}h_1^{-1}g_2h_2g_2^{-1}h_2^{-1}\cdots g_rh_rg_r^{-1}h_r^{-1}=1\in G,$$ then the product
$w_{g_1,h_1}w_{g_2,h_2}\cdots w_{g_r,h_r}\in K^{\times}$ will be a scalar which will be an invariant of $W$, and will thus be contained in $K_0$.

We can understand this invariant by the Hopf formula and the universal coefficients theorem (see Section 2 of \cite{AHN} for more details).
Let $F$ be the free group with generators $x_g$, and let $R$ be the kernel of the homomorphism $F\rightarrow G$ $x_g\mapsto g$.
Then the Hopf formula says that the Schur multiplier $M(G):=H_2(G,\Z)$ is isomorphic with $$([F,F]\cap R )/ [F,R],$$ and the universal coefficients Theorem implies that 
$$H^2(G,K^{\times})\cong Hom_{\Z}(M(G),K^{\times})$$ (we use here the fact that $K$ is algebraically closed).
The cocycle $\alpha$ thus induces a homomorphism $\ti{\alpha}:[F,F]\cap R\rightarrow K^{\times}$ which vanishes on $[F,R]$. 
If $t:=[x_{g_1},x_{h_1}]\cdots [x_{g_r},x_{h_r}]\in [F,F]\cap R$ then a direct calculation shows that $\ti{\alpha}(t) = w_{g_1,h_1}\cdots w_{g_r,h_r}$.

We thus see that the image of $\ti{\alpha}$ is contained in $K_0$.
Since $G$ is finite, $\alpha$ is equivalent to a cocycle whose values are roots of unity. The image of $\ti{\alpha}$ is therefore generated by some root of unity $\mu$.

We will now describe the generic form of $W$. We will show that $W$ has a form over $\Q(\mu)$ and therefore $K_0=\Q(\mu)$.
Notice that already in $\C_W$ we can write $W=(\oplus_{g\in G'}) W_g\oplus (\oplus_{g\notin G'}W_g)$, and the first direct summand is isomorphic in $\C_W$ with $\one^{|G'|}$.
The group $G/G'$ is finite abelian, and we can therefore write $$G/G'\cong \langle \bar{x_1}\rangle\cdots \langle \bar{x_r}\rangle.$$
We write $n_i$ for the order of $\bar{x_i}$.
We consider the object $M=W_{x_1}^*\oplus W_{x_2}^*\oplus \cdots\oplus W_{x_r}^*$.
We let $f_i$ be a basis element for $W_{x_i}^*$ for every $i$. 
We consider the localization $A=Sym(M)_f$ where $f=f_1f_2\cdots f_r$.
Then $A$ is a classifying algebra for $W$.
Indeed, $A\cong A\ot W_{x_i}$ for every $i$.
We also have that $W_g\cong \one$ for every $g\in G'$, and therefore $A\ot W_g\cong A$ for every $g\in G'$.
Since the elements $x_i$ and the elements of $G'$ generate $G$, we can use the multiplication in $W$ to get an isomorphism $A\ot W_g\cong A$ for every $g\in G$,
and therefore $A\ot W\cong A^{|G|}$ as desired. Since the category $\C_W$ is semisimple, $A$ is a classifying algebra.

For every $i$, we can write $U_{x_i}^{n_i} = c_iw_{g^i_1,h^i_1}\cdots w_{g^i_t,h^i_t}$ for some elements $g^i_j,h^i_j\in G$, $c_i\in K^{\times}$.
We can change $U_{x_i}$ by a scalar and we can therefore assume that $c_i=1$ for every $i$. 
The generic form $\ti{W}$ will then be generated by the elements $U_gU_hU_{g^{-1}}U_{h^{-1}}$ and the elements $f_iU_{x_i}$. 
The base ring will be $B_W=K_0[(f_i^{n_i})^{\pm 1}]$.
By taking the algebra generated by $U_{x_i}$ and $w_{g,h}$ we get a form for $W$ defined already over $\Q(\mu)$. This shows that $K_0=\Q(\mu)$.

In \cite{AHN} Aljadeff Haile and Natapov have defined an algebra $U_G$ which they call the universal $G$-graded algebra.
They have defined the algebra using the polynomial graded identities of $W$, 
and have described it as the subalgebra of $W\ot_K K[t_g^{\pm 1}]_{g\in G}$ generated over $\Q$ by the elements $U_g\ot t_g$.
The algebra $U_G$ will be the generic form associated to the classifying algebra $Sym(W^*)_f$ where $f=\prod_g f_g$. The resulting base ring will be bigger though.
Since we choose $M$ instead of $W^*$, we get a smaller base algebra of smaller Krull dimension: the rank of the abelian group $G/G'$ instead of $|G|$.
\end{subsection}

\begin{subsection}{Taft Algebras}
We begin by recalling the definition: $$H_n=k<g,x>/(g^n-1,x^n,gxg^{-1}-\zeta x)$$ where $\zeta$ is a primitive $n$-th root of unity. The comultiplication is given by
$\Delta(g)=g\otimes g$ and $\Delta(x)=x\ot 1 + g\ot x$. The Hopf algebra $H$ is defined over $k:=\Q(\zeta)$.
Let $W=\, ^{\alpha}H$ be defined over $K$.
We have an isomorphism of $H$-comodules $H\cong W$. We write the image of $h\in H$ by $\til{h}\in W$.
We are going to use the fundamental category of $W$ in order to classify all two $H$-cocycles over $K$ up to equivalence.

We will do the following: we will use the maps we have in the category $\C_W$ in order to decompose $W$ 
as the direct sum of weight spaces with respect to some commutative subalgebra of $H^*$
and another commutative subalgebra of $W$.
This decomposition will give us an invariant $b\in K$.
We will then show that $b$ already defines the isomorphism type of $W$, 
and that the assignment $W\mapsto b$ gives us a one-to-one correspondence between the field $K$ and the different possible isomorphism types of cocycles 
(in case $K$ is algebraically closed).
We will then construct a generic form and describe the different forms.

We begin by considering the element $\gamma\in H^*$ given by $$\gamma(g^ix^j)=\zeta^i\delta_{j,0}.$$
The element $\gamma$ generates the group of group-like elements in $H^*$.
In particular, $\gamma :W\rightarrow W$ is an algebra map.
We thus write $W=\oplus W_i$, where $W_i$ is the subspace upon which $\gamma$ acts by $\zeta^i$.
Since the action of $\gamma$ does not depends on $\alpha$, we see that
$$W_i=span\{\widetilde{g^ix^j}\}_j.$$
Notice that the decomposition $W=\oplus_i W_i$ takes place in $\C_W$. 
We now consider another element in $H^*$, namely $\xi$ which is given by
$$\xi(g^ix^j)=\delta_{j,1}.$$
Then one can show that $\Delta(\xi)=1\ot \xi + \xi\ot \gamma^{-1}$ and that $\gamma \xi\gamma^{-1} = \zeta\xi$.
But this implies that $\xi(W_i)\subseteq W_{i+1}$.
We consider now also $Ker(\xi)=T_1$. It is easy to see, again, since the action of $H^*$ does not depend on $\alpha$, 
that this is just the space spanned by the powers of $\til{g}$.
Since $\gamma\xi\gamma^{-1}=\zeta\xi$, we have that $\gamma(T_1)=T_1$ and $W_i\cap T_1= K\til{g}^i$.
Now for every $i$ we have the map $\phi_i:W\rightarrow (W_1\cap T_1)^*\ot(W_1\cap T_1)\ot W\rightarrow (W_1\cap T_1)^*\ot W$ 
which is the composition of $coev_{(W_1\cap T_1)}\ot 1$
with the map $a\ot b\ot c\mapsto a\ot (bc-\zeta^i cb)$ (the multiplication here is the multiplication in $W$).
We have that $Ker(\phi_i)=\{y| \til{g} y \til{g}^{-1}=\zeta^i y\}$ (we have used the coevaluation and the dual of $W_1\cap T_1$ so that $Ker(\phi_i)$ will be a subspace of $W$).
Since $\til{g}^n\in K^{\times}$, conjugation by $\til{g}^n$ is trivial,
and all the eigenvalues of conjugation by $\til{g}$ are of the form $\zeta^i$ for some $i$.
We therefore have a second direct sum decomposition, $$W_i=\oplus_{j} W_{i,j}$$
where $W_{i,j}= W_i\cap Ker(\phi_j)$ (to see that this is a direct sum we need to check that conjugation by $\til{g}$ stabilizes $W_i$, but this is immediate).
A direct calculation shows that we have $\xi(W_{i,j})\subseteq W_{i+1,j-1}$.
We claim now the following:
\begin{lemma}
 We have $dim_K(W_{i,j})=1$ for each $i$ and for each $j$.
\end{lemma}
\begin{proof}
We know that the kernel of $\xi$ is $span_K\{\til{g}^i\}_i$. We have that $\til{g}^i\in W_{i,0}$.
Also, we know that as an $H$-comodule, $W$ is isomorphic to $H$.
This means that $W$ is isomorphic to $H$ as an $H^*$-module.
But this means that there are elements $t_{i,1}\in W$ such that $\xi(t_{i,1})=\til{g}^i$.
Now, $t_{i,1}$ will be a sum of an element $s_{i,1}\in W_{i-1,1}$ and an element in $T_1$. Since $T_1\cap W_{i-1,1}=0$, the element $s_{i,1}$ is well defined.
We thus have, without loss of generality, that $t_{i,1}\in W_{i-1,1}$.
We can now continue in a similar fashion: assuming that $n>2$, there are also element $t_{i,2}\in W$ such that $\xi(t_{i,2})=t_{i,1}$
(we use here the fact that for every $j<n$ we have that $Ker(\xi^j)=im(\xi^{n-j})$).
Again, $t_{i,2}$ will be the sum of an element in $W_{i-2,2}$ and an element in $T_1$. Since $T_1\cap W_{i-2,2}=0$, this element is uniquely defined.
We thus assume, without loss of generality, that $t_{i,2}\in W_{i-2,2}$.
We continue up to $n-1$ in this way.
We thus got, for every $i=0,\ldots,n-1$, $j=0,\ldots,n-1$, a nonzero element $t_{i+j,j}\in W_{i,j}$ (indices are modulo $n$).
So for each $i,j$ we have that $dim_K(W_{i,j})\geq 1$. But then the dimensions already sum up to $n^2$, which is the dimension of the algebra $W$,
so we must have an equality.
\end{proof}
Notice that the last lemma says something stronger. We can deduce that the restriction of $\xi$ to $W_{i,j}$ gives us an isomorphism 
$\xi:W_{i,j}\rightarrow W_{i+1,j-1}$ for every $j\neq 0$.
In particular, consider $1\in W_{0,0}$. There is a unique element $t\in W_{-1,1}$ such that $\xi(t)=1$.
Since the multiplication respects the double grading on $W$ we have that $t^n\in W_{0,0}=span_K\{1\}$.
So $t$ is an invariant vector, and $t^n=b\in K$ is an invariant of the $H$-comodule algebra $W$.
We have that $t^i\in W_{-i,i}$ for every $i< n$. We claim that $t^i\neq 0$ for every $i<n$.
We have that $\xi(t^i)= (1+\zeta+\cdots + \zeta^{i-1})t^{i-1}$, so this follows easily by induction.
This already implies that for every $i=0,\ldots, n-1$, $t^i$ spans $W_{-i,i}$.
The restriction to the subalgebra generated by $\tilde{g}$ will give us that $\tilde{g}^n=a\in K$ is non zero (because this is a two-cocycle on the group $C_n$).
We thus have that for every $i=0,\ldots,n-1$ and every $j=0,\ldots,n-1$ $\til{g}^jt^i$ spans $W_{j-i,i}$.

It follows that our algebra $W$ has a basis given by $\{\til{g}^it^j\}_{i,j}$ subject to the relations $$\til{g}^n=a, t^n=b,$$ $$\til{g}t\til{g}^{-1}=\zeta t.$$
The coaction of $H$ is given by:
$$\rho(\til{g})=\til{g}\ot g$$ $$\rho(t)= t\ot g^{-1} + 1\ot g^{-1}x$$ 
Since $\til{g}$ and $x$ are generators of $W$, this determines $\rho$.
This shows that the isomorphism type of $W$ depends only on the pair $(a,b)$.
Moreover, it is not hard to show that for any pair $(a,b)$ we get an $H$-comodule algebra in this way.
It is possible, however, that different values of $a$ will give us isomorphic algebras (we have already seen that $b$ is an invariant of the algebra).
Indeed, if we change $\til{g}$ to be $x\til{g}$ for some $x\in K$, then we replace $a$ by $ax^n$. 
Since $K$ is assumed to be algebraically closed, we can assume without loss of generality that $a=1$. 
Thus, over $K$ the equivalence classes of cocycles on $H$ are in one-to-one correspondence with elements of $K$, 
and the field of invariants for the cocycle which corresponds to $b$ is $K_0=\Q(\zeta,b)$.

We construct now a generic form for the algebra $W$ which corresponds to $(1,b)$.
We take $A=Sym((W_{1,0})^*)_f$, where $f$ is a basis element of $(W_{1,0})^*$.
The resulting base algebra will be $B_W=\Q(\zeta,b)[f^{\pm n}]$ (so it will be a Laurent polynomial ring in one variable $a:=f^n$).
The generic form will then be generated over $B_W$ by $\til{g}$ and $t$, subject to the relations written above and with the $H$-coaction written above. 
The only difference is that now $a$ will be a generic (invertible) element, and not a specific element of the ground field.
It is easy to see that for any extension field $\Q(\zeta,b)\subseteq K_1$ the forms correspond to $(a,b)$ and $(a',b)$ will be isomorphic if and only if 
$a/a'\in (K_1^{\times})^{n}$,
and thus the different forms are in one-to-one correspondence with the group $K_1^{\times}/(K_1^{\times})^n$
\end{subsection}

\begin{subsection}{Products of Taft Hopf algebras}
We shall study now the same question for the tensor product of Taft Hopf algebras.
Assume that for every $i=1,\ldots,z$, $H_i$ is a Taft Hopf algebra of dimension $n_i^2$.
Let $n=l.c.m(n_i)$, and let $\zeta$ be a primitive $n$-th root of unity.
Then $H_i$ is generated by $g_i,x_i$, subject to the relations $g_i^{n_i}-1=x_i^{n_i}=0$, $g_ix_ig_i^{-1}=\zeta^{c_i}x_i$
where $\zeta^{c_i}$ is a primitive $n_i$-th root of unity.
The comultiplication in $H_i$ is given by $\Delta(g_i)=g_i\otimes g_i$ and $\Delta(x_i)=x_i\otimes 1 + g_i\otimes x_i$.
We write $H=\bigotimes_{i=1}^z H_i$. We shall classify all algebras of the form $^{\alpha}H$ and describe the generic forms.

So let $\alpha$ be a two-cocycle on $H$. 
To begin with, the restriction of $\alpha$ to each of the algebras $H_i$ will give us an $H_i$ comodule algebra.
By the results of the last subsection, 
we can assume that this algebra is generated by the elements $\tilde{g_i}$ and $t_i$
subject to the relations $t_i^{n_i}=b_i$, $\til{g_i}^{n_i}=a_i$ and $\til{g_i}t_i\til{g_i}^{-1}=\zeta^{c_i}t_i$.
The only thing that we need in order to understand $^{\alpha}H$ as well, is to understand how the subalgebras $^{\alpha}H_i$ commute with one another.

We begin by looking on the restriction of $\alpha$ to the group algebra of $G=\langle g_1,g_2,\ldots,g_z\rangle$.
This is a finite abelian group, and we understand well the elements in $H^2(G,K^{\mult})$.
The invariants of $\alpha|_G$ will be given by the scalars $\zeta^{b_{ij}}$ such that $\tilde{g_i}\til{g_j}=\zeta^{b_{ij}}\til{g_j}\til{g_i}$.
Notice that $\zeta^{b_{ij}}$ must be an $n_i$ and $n_j$ root of unity. In other words, $n|c_ib_{ij}$ and $n|c_jb_{ij}$. 
Notice also that $b_{ii}= b_{ij}+b_{ji}=0 \textrm{ mod } n$.

The cocycle $\alpha|_G$ is completely determined by giving these scalars together with the scalars $a_i$

We have that $\rho(t_j)=t_j\otimes g_j^{-1}+ 1\otimes g_j^{-1}x_j$ and $\rho(\til{g_i})=\til{g_i}\otimes g_i$ and therefore
$$\rho(\til{g_i}t_j\til{g_i}^{-1})=\til{g_i}t_j\til{g_i}^{-1}\otimes g_j^{-1} + 1\otimes g_ig_j^{-1}x_jg_i^{-1} = $$
$$ \til{g_i}t_j\til{g_i}^{-1}\otimes 1 + 1\otimes g_j^{-1}x_j$$
So if we write $t_{ij}=\til{g_i}t_j\til{g_i}^{-1}-t_j$ we have that $\rho(t_{ij})=t_{ij}\ot g_j^{-1}$
and therefore $t_{ij}\in K\cdot \til{g_j}^{-1}$.
By conjugating $t_{ij}$ with $\til{g_j}$ we see that $\til{g_j}t_{ij}\til{g_j}^{-1}=\zeta^{c_j}t_{ij}$.
But since $t_{ij}\in K\cdot \til{g_j}^{-1}$ we also have that $\til{g_j}t_{ij}\til{g_j}^{-1}=t_{ij}$.
This implies that $t_{ij}=0$ and therefore $\til{g_i}t_j\til{g_i}^{-1}=t_j$.

The last thing we need to understand is the commutation relations between $t_i$ and $t_j$ for $i\neq j$.
Consider the element $t_it_j - t_jt_i$.
A direct calculation shows that $\rho(t_it_j-t_jt_i)=(t_it_j-t_jt_i)\otimes {g_i}^{-1}{g_j}^{-1}$ and therefore $t_it_j-t_jt_i=\lambda_{i,j}\til{g_i}^{-1}\til{g_j}^{-1}$.
We call the indices $i$ and $j$ \textit{connected}, and we write $i\con j$ if $\lambda_{ij}\neq 0$.
The fact that $i$ and $j$ are connected has some consequences:
\begin{lemma}
Let $i$ and $j$ be connected indices.
Then it holds that $b_{ij}=-c_i=c_j$ \textit{ mod n} and $b_{ik} + b_{jk} = 0$ \textit{ mod n} for every third index $k$ which is not $i$ nor $j$.
\end{lemma}
\begin{proof}
We conjugate the equation $t_it_j - t_jt_i = \lambda_{ij}\til{g_i}^{-1}\til{g_j}^{-1}$ by $\til{g_i}$, $\til{g_j}$ and $\til{g_k}$.
The result follows from the fact that $\lambda_{ij}$ is a nonzero scalar. 	
\end{proof}
We can now write the algebra $W$ and the coaction explicitly:
$W$ is generated by the elements $\til{g}_i,t_i$ subject to the following list of relations:
$$\til{g_i}^{n_i}=a_i,\quad t_i^{n_i}=b_i,\quad \til{g_i}t_i\til{g_i}^{-1}=\zeta^{c_i}t_i$$
$$\til{g_i}\til{g_j}=\zeta^{b_{ij}}\til{g_j}\til{g_i},\quad \til{g_i}t_j\til{g_i}^{-1}= t_j \textrm{ for }i\neq j$$
$$t_it_j-t_jt_i = \lambda_{ij}\til{g_i}^{-1}\til{g_j}^{-1}$$
$$\rho(\til{g_i})=\til{g_i}\ot g_i$$
$$\rho(t_i) = t_i\ot g_i^{-1} + 1\ot g_i^{-1}x_i.$$
As we have seen in the study of cocycles over group algebras and over Taft algebras, 
the cohomology class of the cocycle $\alpha$ determines the scalars $b_i$ and $\zeta^{b_{ij}}$.
The presence of the scalars $\lambda_{ij}$ makes it harder to understand what are the other invariants of the cocycle.
To do this, notice first that if $i\con j$ then $$\Lambda_{ij}:=(t_it_j-t_jt_i)^{n_i} = \pm \lambda_{ij}^{n_i}a_i^{-1}a_j^{-1}$$
is an invariant of $\alpha$.
Second, if $i_1\con i_2\cdots \con i_m\con i_1$ then $$\Lambda_{i_1i_2\ldots i_m}= \lambda_{i_1i_2}\lambda^{-1}_{i_2i_3}\cdots \lambda^{(-1)^{m+1}}_{i_mi_1}$$
is also an invariant of $\alpha$ (notice that the only possibility in which $m$ is odd is if $c_i=n/2$ for $i=i_1,\ldots i_m$).
It is possible to show that over an algebraically closed field we can find a cocycle equivalent to $\alpha$ which can be written in terms of the invariants
$b_i$, $\zeta^{b_{ij}}$, $\Lambda_{ij}$ and $\Lambda_{i_1i_2\ldots i_m}$ (by altering $\tilde{g_i}$ by a nonzero scalar). 
We summarize our discussion in the following proposition:
\begin{proposition}
The cocycle $\alpha$ is determined (up to equivalence) by the scalars $a_i$, $b_i$, $\zeta^{b_{ij}}$, $\lambda_{ij}$.
The cocycle $\alpha$ determines the scalars $b_i$, $\zeta^{b_{ij}}$, $\Lambda_{ij}$ and $\Lambda_{i_1i_2\ldots i_m}$.
These scalars satisfy the relations: $\zeta^{b_{ij}}$ is an $n_i$ root of unity, $\zeta^{b_{ij}+b_{ji}}=1$, $\zeta^{b_{ii}}=1$,
and if $\lambda_{ij}\neq 0$, then $\zeta^{b_{jk}+b_{ik}}=1$ and $\zeta^{b_{ij}}=\zeta^{c_j}$.
Moreover, any collection $(b_i,\zeta^{b_{ij}},\lambda_{ij})$ 
of scalars which satisfy the above relations will give us a cocycle on $H$.
\end{proposition}
\begin{proof}
The fact that $b_i$ and $\zeta^{b_{ij}}$ are invariants of the cocycle $\alpha$ follows from our discussion on group algebras and on Taft algebras. 
Since we wrote the algebra in terms of the scalars in the proposition, 
the collection of scalars $b_i$, $a_i$, $\zeta^{b_{ij}}$, $\lambda_{ij}$ determines the equivalence class of the cocycle $\alpha$ (also over a nonalgebraically closed field).

In the other direction, if we have such a collection of scalars which satisfies the condition of the proposition, 
it is possible to construct an algebra with these relations and coaction. 
The only nontrivial part is to show that this algebra is really of the form $^{\alpha}H$ 
(a priori, it is possible that the relations we have will define the trivial algebra, for example).
The algebra can be constructed by Ore extensions and crossed products, 
and we can prove by induction on the number of factors $z$ that the algebra is really of the form $^{\alpha}H$.
\end{proof}
We will now construct the generic form of $^\alpha{H}$.
As before, we will take $M=\oplus_i (K\cdot \til{g_i})^*$ and $A=Sym(M)_f$ where $f=\prod f_i$. ($f_i$ is the dual basis for $\til{g_i}$ for the space $K\cdot \til{g_i}$).
Then the base algebra will be $B=\Q(\zeta,b_i,\Lambda_{ij},\Lambda_{i_1\ldots i_m})[a_1^{\pm 1},a_2^{\pm 1},\ldots a_z^{\pm 1}]$, where $a_i=f_i^{n_i}$,
and the generic form will be exactly the algebra written above (the only difference is that now $a_i$ are generic elements, and not elements of the ground field).
Notice that we can take even a smaller algebra: indeed, if $i\con j$ then the vector $t_it_j-t_jt_i = \lambda_{ij}\til{g_i}^{-1}\til{g_j}^{-1}$ will be contained in any form.
We can then take $M=\oplus_i (K\cdot \til{g_i})^*$ where we take only one index $i$ from each equivalence class of the equivalence relation generated by $\con$.
\end{subsection}
\subsection*{Acknowledgements}
The author was supported by the Danish National Research Foundation (DNRF) through the Centre for Symmetry and Deformation.
The author is grateful to Eli Aljadeff, Apostolos Beligiannis, Julien Bichon, Pavel Etingof, Gaston Garcia, Christian Kassel, Bernhard Keller, 
Henning Krause, Akira Masuoka and Jan Stovicek for guidance and for fruitful discussions.
\end{section}

\end{document}